\documentclass{amsart}

\usepackage{tikz-cd}

\usepackage{enumerate}

\usepackage{macros}
\usepackage[
final
]{showlabels}

\usepackage[T1]{fontenc}
\usepackage[
  style=alphabetic,
  citestyle=alphabetic,
  maxnames=100, 
  doi=false,
  isbn = false,
  giveninits = true, 
]{biblatex}
\DeclareNameAlias{default}{family-given}

\usepackage[
  pdfauthor={Joshua Mundinger, Shubhankar Sahai},
  pdftitle={Tannakian reconstruction in derived algebraic geometry},
]{hyperref}

\title{
Tannakian reconstruction in derived algebraic geometry
}
\begin{document}

\begin{abstract}
We prove analogues of Tannakian reconstruction theorems of Lurie and Bhatt--Halpern-Leistner in \emph{derived} algebraic geometry, where the basic geometric objects are spectra of animated rings rather than $\mathbb{E}_\infty$-rings. 
In this setting, symmetric monoidal $\infty$-categories are replaced by the $\Theta$-categories of Nuiten--Toën. These are enhancements of symmetric monoidal $\infty$-categories which capture the strict commutativity structure on animated rings.
\end{abstract}
\author{Joshua Mundinger}
\address{Joshua Mundinger \\
    Department of Mathematics \\
	University of California, Berkeley \\
    970 Evans Hall \\
    Berkeley, CA\\
    94720\\
    USA}
\email{mundinger@berkeley.edu}

\author{Shubhankar Sahai}
\address{Shubhankar Sahai \\
    Department of Mathematics \\
	University of California San Diego \\
    2985 Muir Ln\\ 
    La Jolla, CA \\
    92093\\
    USA}
\email{shubhankar.sahai.72@gmail.com}

\date{August 2, 2026}

\maketitle

\section{Introduction}

The goal of this paper is to prove Tannakian reconstruction theorems in derived algebraic geometry, that is, algebraic geometry over \emph{animated} (i.e. simplicial commutative) rings. 

\subsection{Background}\label{ssec: background}

Tannakian reconstruction concerns the extent to which an algebraic stack is determined by its category of quasi-coherent sheaves. The theory originates in the work of several authors, including the Gabriel-Rosenberg reconstruction theorem for quasi-separated schemes \cite{gabrielcategories,Ros95,Bra18} and Saavedra Rivano's thesis on representations of affine group schemes \cite{Saa72,DM82}. 

Lurie \cite{lurie04} unified these perspectives and brought attention to the relative setting: instead of only assigning to a stack $X$ the abelian category $\QCoh(X)^\heartsuit$ of quasi-coherent sheaves,
one should also consider assigning to a morphism $f: S \to X$ the pullback functor $f^*: \QCoh(X)^{\heartsuit} \to \QCoh(S)^\heartsuit$.
When we remember that $f^*$ is symmetric monoidal, Lurie showed that the association $f \mapsto f^*$ is fully faithful for a reasonable class of 
stacks. In \cite{luriedagviii}, Lurie extended the theory to \emph{spectral} algebraic geometry, where the basic geometric objects are spectra of $\Einfty$-rings, which necessitates the use of stable $\infty$-categories instead of abelian categories. The resulting theory of Tannaka duality has many applications, even to classical algebraic geometry; see for example \cite{BhattAlgebraizationTannaka}. At the most basic level, Lurie's theory relies on the observation that assigning to an $\Einfty$-ring $R$ the symmetric monoidal $\infty$-category $\Mod_R$ is fully faithful.

From the perspective of derived algebraic geometry outside of characteristic zero, it is desirable to work with animated (i.e.\ simplicial commutative) rings instead of $\Einfty$-rings. For example, in spectral algebraic geometry, unless $R$ contains $\mathbb{Q}$, the `flat' affine line $\Spec(R[t])$ is not smooth, while the `smooth' affine line $\Spec (R\{t\})$, where $R\{t\}$ is the free $\Einfty$-algebra  on one generator, is not flat. 
Working with animated rings fixes these and other issues while still allowing for robust constructions in derived algebraic geometry, such as derived intersections and deformation theory.

However, to reconstruct an animated ring $A$, the data of its symmetric monoidal $\infty$-category $\Mod_A$ is not enough. Indeed, the symmetric monoidal $\infty$-category $\Mod_A$ only depends on the $\Einfty$-ring underlying $A$, and forgetting the animated structure is not a fully faithful functor away from characteristic $0$. Philosophically, the problem arises because a symmetric monoidal $\infty$-category is an $\Einfty$-notion, 
while an animated ring has stricter commutativity requirements.

Nuiten and Toën recently proposed a fix with the notion of \emph{$\Theta$-category} \cite{NT25}. A $\Theta$-category is, roughly speaking, a symmetric monoidal $\infty$-category together with a monad refining the $\Einfty$-monad.
The assignment of $A$ to an appropriate $\Theta$-category $\Mod^{\LSym}_A$ is fully faithful, essentially because we carry along the $\infty$-category of derived $A$-algebras by construction \cite[Proposition 1.12]{NT25}.

Recent works in $p$-adic Hodge theory \cite{Dri24Prismatization, Bhatt22}, Hochschild homology \cite{universal_hkr, Mun24}, and affine stacks \cite{MM25},
which take place in derived geometry in positive or mixed characteristic, suggest that a systematic study of Tannakian theorems in derived algebraic geometry might be fruitful.
In particular, see \cite[§4]{Mun24} for an application of this Tannakian formalism in positive characteristic to the Hochschild-Kostant-Rosenberg spectral sequence for Hochschild homology.
\subsection{Results}\label{ssec: results}

Nuiten and Toën's $\Theta$-categories naturally assemble into an $(\infty,2)$-category $\thetaCat$ whose construction we recall in §\ref{section: theta-categories}.
If $X$ is a prestack\footnote{i.e., a presheaf of $\infty$-groupoids on the opposite $\infty$-category of that of animated rings.}, there is a $\Theta$-category $\QCohL(X)$, functorial in $X$, which encodes the usual symmetric monoidal $\infty$-category $\QCoh(X)$ together with the derived symmetric algebra monad $\LSym$ of Raksit and Mathew (see \cite{Rak20, BM25}). For prestacks $X$ and $Y$ over a prestack $Z$, let $\Fun_\Theta^{\QCohL(Z)}(\QCohL(X),\QCohL(Y))$ be the $\infty$-category of morphisms of $\Theta$-categories 
$\QCohL(X) \to \QCohL(Y)$ under $\QCohL(Z).$

In this setting, we are able to prove a derived version of Lurie's full faithfulness result \cite[Proposition 3.3.11]{luriedagviii}; moreover, our result is stated relative to a base:

\begin{theoremalpha}[Theorem \ref{theorem: quasi-affine diagonal and full faithfulness}]\label{maintheorem: fully faithful}
    Let $Z$ be a prestack with quasi-affine diagonal.
    Let $X$ be a prestack with quasi-affine diagonal equipped with a morphism $X \to Z$.
    Let $Y \to Z$ be a morphism of prestacks.
    Then the association $f \mapsto f^*$ gives a fully faithful embedding
    \[ 
        \Maps_Z(Y,X) \to \Fun_\Theta^{\QCohL(Z)}(\QCohL(X),\QCohL(Y)).
    \]
\end{theoremalpha}
When the target stack $X$ has a faithfully flat cover by an affine scheme and an affine diagonal, Lurie was able to identify the essential image in the spectral setting as those colimit-preserving symmetric monoidal functors which preserve connective objects and flat objects \cite[Theorem 3.4.2]{luriedagviii}, again in the relative setting. 
\begin{theoremalpha}[Theorem \ref{theorem: essential image for relative geometric stack}]\label{maintheorem: essential image}
    Let $Z$ be an fpqc sheaf with quasi-affine diagonal.
    Let $X \to Z$ be a relative geometric derived stack and let $Y \to Z$ be a morphism of prestacks. Then the association $f \mapsto f^*$ gives a fully faithful embedding
    \[\Maps_Z(Y,X) \to \Fun_\Theta^{\QCohL(Z)}(\QCohL(X),\QCohL(Y))\]
    with essential image spanned by morphisms of $\Theta$-categories preserving connective objects and flat objects.\footnote{In Lurie's work, symmetric monoidal functors are also required to be colimit-preserving. Morphisms of $\Theta$-categories preserve colimits by definition.} 
\end{theoremalpha}

In \cite[Remark 8.5]{AS25}, Antieau and Stefanich state a version of Theorem \ref{maintheorem: essential image} under the additional hypothesis of $\DAlg$-descendability of a certain morphism of derived algebras. They only consider the assignment $X \mapsto \DAlg(X)$ of a prestack to its category of derived algebras, which necessitates the stronger hypothesis.

Bhatt \cite{BhattAlgebraizationTannaka} and Bhatt and Halpern-Leistner \cite{BH17} also considered the essential image of the association $f \mapsto f^*$, seeking to relax the condition of preserving flat objects, which is difficult to check in practice. 
When $\QCoh(X)$ is compactly generated, there is an improved version of Theorem \ref{maintheorem: essential image}, as Bhatt and Halpern-Leistner discovered in the spectral setting \cite[Theorem 4.1]{BH17}.

\begin{theoremalpha}[Theorem \ref{theorem: relative compactly generated}]\label{maintheorem: relative compactness}
    Let $Z$ be an fpqc sheaf with affine diagonal. Let $X \to Z$ be a relative quasi-geometric stack such that $\QCoh(X)$ is compactly generated and let $Y\to Z$ be a morphism of prestacks. 
    Then the association $f \mapsto f^*$ gives a fully faithful embedding
    \[\Maps_Z(Y,X)\to \Fun_\Theta^{\QCohL(Z)}(\QCohL(X),\QCohL(Y))\]
    with essential image spanned by morphisms of $\Theta$-categories preserving connective objects.
\end{theoremalpha}

If $X$ is a quasi-compact quasi-separated derived algebraic space, then $\QCoh(X)$ is compactly generated; in this case preservation of connective objects is automatic, so that $f \mapsto f^*$ is an equivalence. This follows the classical case due to Bhatt \cite[Theorem 1.5]{BhattAlgebraizationTannaka}. See Corollary \ref{cor: ff embedding algsp} for more details.

There is a similar improvement when the target stack is locally Noetherian (Theorems \ref{theorem: noetherian geometric} and \ref{theorem: noetherian quasi-geometric}).

\begin{theoremalpha}[Theorem \ref{theorem: noetherian quasi-geometric}]\label{maintheorem: noetherian quasi-geometric}
    Let $X$ be a locally Noetherian quasi-geometric stack and let $Y$ be any prestack. Then the association $f \mapsto f^*$ gives a fully faithful embedding
    \[\Maps(Y,X)\to \Fun_\Theta^{\Mod_{\ZZ}^{\LSym}}(\QCohL(X),\QCohL(Y))\] 
    with essential image spanned by morphisms of $\Theta$-categories preserving connective objects and almost perfect objects.
\end{theoremalpha}

Lastly, we are able to formulate some theorems for derived formal stacks using the relative statements of Theorems \ref{maintheorem: fully faithful} and \ref{maintheorem: relative compactness}. Let $A$ be an animated ring, $J\subset H^0(A)$ be a finitely generated ideal, and $\Spf(A,J)$ be the $J$-adic formal spectrum of $A$.

\begin{theoremalpha}[Theorem \ref{theorem: formal compactly generated}]\label{maintheorem: formal relative compactness}
    Let $X\to \Spf(A,J)$ be a relative quasi-geometric stack such that $\QCoh(X)$ is compactly generated and let $Y \to \Spf(A,J)$ be a morphism of prestacks. Then the association $f \mapsto f^*$ gives a fully faithful embedding
    \[\Maps_{\Spf(A,J)}(Y,X)\to \Fun_\Theta^{(\Mod_A^{\LSym})^\jcomp}(\QCohL(X),\QCohL(Y))\]
    with essential image spanned by morphisms of $\Theta$-categories preserving connective objects.
\end{theoremalpha}
Here $(\Mod_A^{\LSym})^\jcomp$ is a natural $\Theta$-structure on the $\infty$-category $\Mod_A^\jcomp$ of (derived) $J$-complete $A$-modules, obtained by completing the monad $\LSym_A$ on $\Mod_A$. See §\ref{subsec: theta structure formal modules} for details.

\subsection{Summary of the paper}\label{ssec: summary of the paper}

In §\ref{section: preliminaries}, we review preliminary notions in derived algebraic geometry. In §\ref{section: theta-categories}, we recall the notion of $\Theta$-category of Nuiten and Toën \cite{NT25} and construct a $\Theta$-structure on the $\infty$-category of quasi-coherent sheaves on a derived prestack. In §\ref{section: quasi-affine}, we discuss quasi-affine morphisms of prestacks. 
In §\ref{section: lurie's theorems}, we prove derived versions of Lurie's spectral Tannakian theorems. §\ref{section: compactly generated} and §\ref{section: noetherian} give improvements to these theorems when the stack in question has compactly generated derived category or is Noetherian, respectively. 
Finally, §\ref{section: formal} discusses Tannakian reconstruction for formal stacks.

Let us now mention what is not in this paper. We deal only with the question of reconstruction (full faithfulness of $X \mapsto \QCohL(X)$) and not of recognition (the essential image of $X \mapsto \QCohL(X)$); preliminary results on recognition in this setting are in \cite[§2]{NT25}. We also do not consider Tannakian reconstruction for sheaves of $(\infty,n)$-categories for $n>1$, as in Scholze and Stefanich's theory of \emph{gestalten} \cite{gestalten}. Thus, our results do not and cannot apply to certain higher stacks such as $B^2\mathbb{G}_m$,
for which $\QCohL$ is too coarse of an invariant\footnote{$B^2\mathbb{G}_m \to \ast$ induces an equivalence on quasi-coherent sheaves, see \cite[Example 1.13]{gestalten}. The theorems of this article do not apply since the diagonal of $B^2\mathbb{G}_m$ is represented by $B\mathbb{G}_m$, which is not quasi-affine. Conceptually, it is important that the diagonal is relatively weakly $0$-affine in the sense of Gaitsgory; see Remark \ref{rem: full faithfulness and derived affine?} for more details.}. 
Further, the theory of gestalten is at present only formulated in the spectral setting, that is, using $\Einfty$-rings, and it is not clear how to strictify the notion. In \cite{Ste23}, Stefanich used sheaves of $(\infty,1)$-categories to show that for spectral geometric stacks, the hypothesis ``preserves connective objects'' is redundant. It would be interesting to know if the same holds for derived geometric stacks.

\subsection*{Acknowledgments}
The authors thank
Ben Antieau, 
Dima Arinkin,
Bhargav Bhatt,
Adeel Khan,
Bertrand Toën, 
Nathan Wenger 
and
Chris Xu
for useful comments and conversations.

In the course of this work, J.\ M.\ was supported by the National Science Foundation under Award No.\ 2503534.
S.\ S.\ was partially supported by NSF grant DMS-2401536 under Kiran Kedlaya.
Any opinions, findings, and conclusions or recommendations expressed in this material are those of the authors and do not necessarily reflect the views of the National Science Foundation. 

\subsection*{LLM Statement}
ChatGPT 5.6 Sol devised Example \ref{ex: einfty localization but not derived}. The example was verified for accuracy by the authors. 

\section{Preliminaries}\label{section: preliminaries}

We adopt the convention that a category is an $(\infty,1)$-category in the sense of Lurie \cite{HTT}. Similarly, a groupoid will always mean an $\infty$-groupoid. Let $\Spc$ be the category of spaces/($\infty$-)groupoids/anima. By a $2$-category we mean an $(\infty,2)$-category; our reference for $2$-categories is \cite[Appendix D]{HM6ff}. 

Let $\aring$ be the category of animated rings. Let $\Aff := \aring^{op}$  be the category of affine schemes.
In this article, a \emph{prestack} is an accessible functor $X\colon \Aff^{op} \to \Spc$\footnote{
    A functor is accessible if it commutes with $\kappa$-filtered colimits for some regular cardinal $\kappa$.
}. Let $\PreStk$ be the category of prestacks. 
For $X \in \PreStk$, the category of prestacks over $X$ is the comma category $\PreStk_{/X}$. 
It is equivalent to the category of accessible $\Spc$-valued functors on the category $\aring_{/X}$ of animated rings $R$ with a chosen point of $X(R)$. 
When the base is an affine scheme $\Spec(A)$, we let $\PreStk_A := \PreStk_{/\Spec(A)}$.

For $t$-structures, we use cohomological indexing as in \cite{GR17I}. Thus, for an animated ring $A$, the subcategory $\Mod_A^{\leq 0} \subset \Mod_A$ of connective $A$-modules consists of those $A$-modules $M$ with $H^i(M) = 0$ for $i > 0$. 
The suspension $[1]$, normalized with respect to the shift functor on the category of spectra, has the effect $H^i(M[1])=H^{i+1}(M)$.

\subsection{Geometric and quasi-geometric stacks}

We now define the class of stacks for which we will later calculate the essential image of $f \mapsto f^*$.

Recall (cf. \cite[§3.4]{luriedagviii}) that a \emph{geometric stack} is an fpqc sheaf $X$ with affine diagonal admitting a faithfully flat morphism from an affine scheme.

\begin{definition}\label{defn: relative geometric stack}
    A \emph{relative geometric stack} is a morphism of fpqc sheaves $X \to Y$ where
    \begin{enumerate}
        \item the diagonal $X \to X \times_Y X$ is affine;
        \item there exists an affine morphism $X' \to Y$ and a faithfully flat morphism $X' \to X$ over $Y$.
    \end{enumerate}
\end{definition}
Since $X \to Y$ has affine diagonal, the morphism $X' \to X$ is automatically affine.

Similarly, a \emph{quasi-geometric stack} is an fpqc sheaf $X$ with quasi-affine diagonal admitting a faithfully flat morphism from an affine scheme.
\begin{definition}\label{defn: relative qg stack}
    A \emph{relative quasi-geometric stack} is a morphism of fpqc sheaves $X \to Y$ such that 
    \begin{enumerate}
        \item the diagonal $X \to X\times_Y X$ is quasi-affine;
        \item there exists an affine morphism $X'\to Y$ and a faithfully flat morphism $X' \to X$ over $Y$.
    \end{enumerate}
\end{definition}
It follows that the morphism $X' \to X$ is automatically quasi-affine.

\begin{remark}
    Note that in \cite[§9.1.2]{Lurie-SAG}, Lurie defines a quasi-geometric morphism to be a morphism of prestacks $X \to Y$ whose base change to any affine scheme is a quasi-geometric stack. This is a more general notion than Definition \ref{defn: relative qg stack}, which requires a global presentation as a quotient by a quasi-affine groupoid relative to the base. Our proofs rely on such a presentation.
\end{remark}

The notions of relative geometric stack and relative quasi-geometric stack are stable under base change.

\begin{remark}
    The proofs of this article work when we weaken Definitions \ref{defn: relative geometric stack} and \ref{defn: relative qg stack} to allow $Y$ to just be a prestack provided that we require $X$ to be an fpqc sheaf on $\aring_{/Y}$. However, for ease of exposition, we work with the stronger hypothesis that both $X$ and $Y$ are fpqc sheaves.
\end{remark}

\section{Recollection of $\Theta$-categories}\label{section: theta-categories}

\subsection{The definition of $\Theta$-category}\label{ssec: theta-category}


Informally, a $\Theta$-category is a presentable symmetric monoidal category $\CC$ equipped with a sifted-colimit-preserving monad $M$ and a morphism of monads $\Einfty \to M$ such that $1_\CC \to M(\varnothing)$ and $M(x) \otimes M(y) \to M(x \sqcup y)$ are equivalences for all $x,y \in \CC$ \cite[7]{NT25}.
The conditions on $M$ are equivalent to requiring that the map $\CC^{M} \to \CC^{\Einfty}$ commutes with arbitrary colimits.

We now review the formal definition of $\Theta$-category, following \cite[Definition 1.5]{NT25}. 
Let $\PrR$ be the $2$-category of presentable categories with accessible right adjoint functors.
\begin{definition}
The 2-category $\Prmu$ is the 2-full\footnote{A subcategory of a 2-category $C' \subset C$ is 2-full if $\Hom_{C'}(x,y) \to \Hom_C(x,y)$ is a fully faithful functor for all $x,y \in C'$ \cite[Definition D.1.3(b)]{HM6ff}.}
subcategory of $\Fun(\Delta^1,\PrR)^{op}$ 
with 
\begin{itemize}
    \item objects given by conservative functors $G: T' \to T$ which preserve sifted colimits;
    \item morphisms from $G_1: T'_1 \to T_1$ to $G_2: T'_2 \to T_2$ given by those commutative squares of right adjoint functors 
    \[
\begin{tikzcd}[cramped]
	{T'_2} & {T'_1} \\
	{T_2} & {T_1}
	\arrow[from=1-1, to=1-2]
	\arrow["{G_2}", from=1-1, to=2-1]
	\arrow["{G_1}", from=1-2, to=2-2]
	\arrow[from=2-1, to=2-2]
\end{tikzcd}\]
which are left-adjointable.
\end{itemize}
\end{definition}
The Barr-Beck-Lurie theorem \cite[Theorem 4.7.0.3]{HA} implies that if $G \in \Prmu$, then $G$ is monadic; since $G$ preserves sifted colimits, so does the associated monad.
Thus, $\Prmu$ may be considered as the 2-category of sifted-colimit-preserving monads.
Forgetting the monad and passing to left adjoints gives a conservative functor $\Prmu \to \PrL$, where $\PrL$ is the $2$-category of presentable categories with left adjoint functors.

Let $\CAlg(\PrL)$ be the $2$-category of presentable symmetric monoidal categories and colimit-preserving symmetric monoidal functors.
There is a functor $\Einfty: \CAlg(\PrL) \to \Prmu$ sending a symmetric monoidal category to the $\Einfty$-monad on it.
Let $C$ be the comma object 
\[
\begin{tikzcd}[ampersand replacement=\&]
	C \& \Prmu \\
	\CAlg(\PrL)\& {\Fun(\Delta^1,\PrR)^{op}}
	\arrow[from=1-1, to=1-2]
	\arrow[from=1-1, to=2-1]
	\arrow["r", from=1-2, to=2-2]
	\arrow[Rightarrow, from=2-1, to=1-2]
	\arrow["{r \circ \Einfty}"', from=2-1, to=2-2]
\end{tikzcd}
  \]
where  $r: \Prmu \to \Fun(\Delta^1,\PrR)^{op}$ is the tautological inclusion.
\begin{definition}[\cite{NT25}, Definition 1.5]
   The $2$-category of \emph{$\Theta$-categories} $\thetaCat$ is the full subcategory of the $2$-category $C$ defined above spanned by those objects 
\[ 
\begin{tikzcd}[ampersand replacement=\&]
	{\CC_2^M} \& {\CAlg(\CC_1)} \\
	{\CC_2} \& {\CC_1}
	\arrow[from=1-1, to=1-2]
	\arrow[from=1-1, to=2-1]
	\arrow[from=1-2, to=2-2]
	\arrow["{u_*}", from=2-1, to=2-2]
\end{tikzcd}
\]
such that 
\begin{enumerate}
    \item $u_*: \CC_2 \to \CC_1$ is an equivalence;
    \item $\CC_2^M \to \CAlg(\CC_1)$ commutes with arbitrary colimits.
\end{enumerate}
\end{definition}

For us, the motivating example of a $\Theta$-category comes from derived $\ZZ$-algebras:
\begin{example}\label{ex: theta-category of Z-modules}
    From work of Brantner-Mathew \cite[§4-5]{BM25} and Raksit \cite[§4]{Rak20}, there exists a sifted-colimit-preserving monad $\LSym_\ZZ$ on $\Mod_\ZZ$ preserving $\Mod_\ZZ^{\leq 0}$ such that $\left( \Mod_\ZZ^{\leq 0}\right)^{\LSym}$ is the category of animated rings.
    Further, since polynomial rings are $0$-truncations of free $\Einfty$-algebras, there is a map of monads $\mathbb{E}_{\infty,\ZZ} \to \LSym_{\ZZ}$ \cite[Construction 4.2.19]{Rak20}.
    The induced map $\Mod_{\ZZ}^{\LSym_{\ZZ}} \to \Mod_{\ZZ}^{\Einfty}$ preserves all colimits \cite[Proposition 4.2.27]{Rak20}.
    Hence $(\Mod_\ZZ, \LSym_\ZZ)$ forms a $\Theta$-category.
\end{example}

\begin{lemma}[\cite{NT25}, Remark 1.6]\label{lemma: theta-categories have small limits}
    The $1$-category underlying $\thetaCat$ has small limits,
    and the forgetful functors $\thetaCat \to \CAlg(\PrL)$ and $\thetaCat \to \PrL$ are conservative and preserve limits.
\end{lemma}

\begin{notation}
    Given a $\Theta$-category $\CC$ with monad $M$, we let $\DAlg(\CC)$ be the category of $M$-algebras.
\end{notation}

Call a morphism in $\thetaCat$ a \emph{$\Theta$-functor}. Let $\Fun_\Theta(\CC,\CC')$ be the category of $\Theta$-functors between $\CC,\CC'\in\thetaCat$.
If $\CC \in \thetaCat$, we write $\Fun_{\Theta}^{\CC}$ for the category of morphisms in $\thetaCat_{\CC/}$.

\begin{remark}[Explicit description of $\Theta$-functors. \cite{NT25}, Lemma 1.8]\label{remark: morphisms of theta cats}
It follows from the definition that a $\Theta$-functor $\CC_1 \to \CC_2$ is a triple $(f^*,u_*,\alpha)$ where $f^*: \CC_1 \to \CC_2$ is a colimit-preserving symmetric monoidal functor, $u_*: \DAlg(\CC_2) \to \DAlg(\CC_1)$ is a right adjoint functor, and $\alpha$ is a left-adjointable commutative square in $\PrR$
\[ 
\begin{tikzcd}[ampersand replacement=\&]
	{\DAlg(\CC_2)} \& {\DAlg(\CC_1)} \\
	{\CAlg(\CC_2)} \& {\CAlg(\CC_1)}
	\arrow["{u_*}"', from=1-1, to=1-2]
	\arrow[from=1-1, to=2-1]
	\arrow[from=1-2, to=2-2]
	\arrow["{f_*}"', from=2-1, to=2-2]
\end{tikzcd}
\]
where $f_*$ is induced by the right adjoint to $f^*$.
\end{remark}

\subsection{Module $\Theta$-categories}\label{subsection: module theta-categories}

Given a $\Theta$-category $\CC$ and $B \in \DAlg(\CC)$, 
consider the symmetric monoidal category $\Mod_B(\CC)$ of modules over the underlying $\Einfty$-algebra of $B$.
There is a natural enhancement $\Mod_B(\CC)$ to a $\Theta$-category, constructed as follows \cite[§1.3]{NT25}: 
define $\DAlg(\Mod_B(\CC))$ to be the comma category $\DAlg(\CC)_{B/}$ together with the structure map 
\[ \DAlg(\Mod_B(\CC)) = \DAlg(\CC)_{B/} \to \CAlg(\CC)_{B/} \cong \CAlg(\Mod_B(\CC)) \to \Mod_B(\CC).\]
This defines a $\Theta$-category structure on $\Mod_B(\CC)$, for which we will use the same notation $\Mod_B(\CC)$.

\begin{proposition}[\cite{NT25}, Proposition 1.12]\label{prop: module cat adjunction}
    Given a $\Theta$-category $\CC$, assigning $B\in \DAlg(\CC)$ to $\CC \to \Mod_B(\CC)$ induces a fully faithful left adjoint functor 
    \[ \DAlg(\CC) \to \thetaCat_{\CC/}\]
    whose right adjoint sends $g^*: \CC \to \CC'$ to $g_*\one_{\CC'} \in \DAlg(\CC)$.
\end{proposition}

Thus, for $g: \CC \to \CC'$ a $\Theta$-category under $\CC$,
\[ \Fun_{\Theta}^{\CC}(\Mod_B(\CC), \CC') \simeq \Hom_{\DAlg(\CC)}(B, g_*\one_{\CC'}).\]
In particular, the category $\Fun_\Theta^{\CC}(\Mod_B(\CC),\CC')$ is a groupoid.

The following proposition, which is a relative variant of \cite[Proposition 1.13]{NT25}, gives a criterion for recognition of a module $\Theta$-category via the projection formula. 
\begin{proposition}[\cite{NT25}, Proposition 1.13]\label{prop: recognition of module category}
Let $F: \CC_1 \to \CC_2$ be a morphism of $\Theta$-categories.
    Assume that 
    \begin{enumerate}
        \item on underlying categories, the right adjoint functor $G: \CC_2 \to \CC_1$ is monadic;
        \item $(F,G)$ satisfies the projection formula, that is, for $y \in \CC_2$ and $x \in \CC_1$, the natural morphism 
        \[ x \otimes G(y) \to G(F(x) \otimes y)\]
        is an equivalence in $\CC_1$.
    \end{enumerate}
    Then the counit 
    \[ \Mod_{G(\one_{\CC_2})}(\CC_1) \to \CC_2\]
    of the adjunction of Proposition \ref{prop: module cat adjunction} is an equivalence of $\Theta$-categories under $\CC_1$.
\end{proposition}

\subsection{Localizations in $\Theta$-categories}\label{ssec: localization}

\begin{definition}\label{def: localizations in thetacats}
    Let $\CC$ be a $\Theta$-category. 
    An object $A \in \DAlg(\CC)$ is a \emph{localization} if the unit map $\one_\CC \to A$ is an epimorphism.
\end{definition}
Equivalently, $A$ is a localization if the coproduct map $A \otimes A \to A$ is an equivalence.

\begin{proposition}\label{prop: preservation of epimorphisms}
    Let $U: \DAlg(\CC) \to \CAlg(\CC)$ be the forgetful functor with left adjoint $L: \CAlg(\CC) \to\DAlg(\CC)$.
    \begin{enumerate}
        \item $A \in \DAlg(\CC)$ is a localization if and only if $U(A) \in \CAlg(\CC)$ is.
        \item If $B \in \CAlg(\CC)$ is a localization, then $L(B) \in \DAlg(\CC)$ is a localization.
        \item If $A \in \DAlg(\CC)$ is a localization, then the counit map $LU(A) \to A$ is an equivalence. 
        Thus for any $B \in \DAlg(\CC)$, the map
        \[
            U: \Maps_{\DAlg(\CC)}(A,B) \to \Maps_{\CAlg(\CC)}(UA,UB)
        \]
        is an equivalence.
        \item $U$ induces a fully faithful functor 
        \[U: \{ \text{localizations in }\DAlg(\CC)\} \to \{\text{localizations in }\CAlg(\CC)\}.\]
    \end{enumerate}
\end{proposition}
\begin{proof}
\begin{enumerate}
    \item 
    By definition of $\Theta$-category, the forgetful functor $U$ commutes with finite colimits. Further, $U$ is conservative.
    Thus, if $A \in \DAlg(\CC)$, $A \otimes A \to A$ is an equivalence if and only if 
    \[ U(A) \otimes U(A) \simeq U(A \otimes A) \to U(A)\]
    is an equivalence.

    \item
    If $B \in \CAlg(\CC)$ is a localization, then 
    \[ L(B) \otimes L(B) \simeq L(B \otimes B) \to L(B)\]
    is an equivalence, since left adjoints preserve colimits.

    \item 
    We must show that if $A \in \DAlg(\CC)$ is a localization, then the counit map $\eta_A: LU(A) \to A$ is an equivalence.
    After applying $U$, the composite
    \[ U(A) \to ULU(A) \overset{U(\eta_A)}\to U(A)\]
    is equivalent to the identity of $U(A)$.
    Since $\one_{\CC} \to ULU(A)$ is an epimorphism, this implies that $ULU(A) \to U(A) \to ULU(A)$ is also equivalent to the identity.
    Hence $U(\eta_A)$ is an equivalence.
    Since $U$ is conservative, $\eta_A$ is an equivalence.

    \item follows from (3).
    \end{enumerate}
\end{proof}

The functor 
\[U: \{ \text{localizations in }\DAlg(\CC)\} \to \{\text{localizations in }\CAlg(\CC)\}\]
in Proposition \ref{prop: preservation of epimorphisms} need not be essentially surjective.

\begin{example}\label{ex: einfty localization but not derived}
    Let $A$ be the derived algebra $C^*(BC_2, \mathbb{F}_2)$ of $\mathbb{F}_2$-valued cochains on the classifying space 
    of the cyclic group of order 2. 
    Then $A$ is a derived commutative ring. Let $\CC = \Mod_A(\ModL_{\ZZ})$. Then a localization in $\DAlg(\CC)$ is an epimorphism of derived commutative algebras $A \to B$.
    Now 
    \[ H^\bullet(A) = \mathbb{F}_2[t]\]
    where $t$ is in cohomological degree 1.
    Let $S$ be the multiplicative set $\{1,t,t^2,...\} \subseteq H^\bullet(A)$.
    Then there exists an $\Einfty$-localization $B = S^{-1}U(A)$ of $U(A)$ at $S$; by \cite[Proposition 7.2.3.20]{HA}, one has 
    $H^\bullet(B) = \mathbb{F}_2[t,t^{-1}]$.
    However, $B$ cannot be the underlying $\Einfty$-ring of a derived commutative ring: the square of an element in cohomological degree $-1$ of a derived commutative ring must be zero, as the free such $\ZZ$-algebra is $\ZZ \oplus \ZZ[1]$ by a result of Illusie \cite[Chapitre I, 4.3.2]{CotangentComplexI}.
    On the other hand, $t^{-2} \neq 0$.
    
    This example was suggested by ChatGPT 5.6 Sol and verified by the authors.
\end{example}

In §\ref{ssec: classification of qaff}, we will study localizations in $\DAlg$ which are also compact.
To be compact, it is sufficient for the underlying $\Einfty$-algebra to be compact:

\begin{lemma}\label{lem: detecting compacts in dalg}
    Let $\CC$ be a $\Theta$-category. Suppose $A \in \DAlg(\CC)$ is a localization.
    If $U(A)$ is compact in $\CAlg(\CC)$,
    then $A$ is compact in $\DAlg(\CC)$.
\end{lemma}
\begin{proof}
    By definition of $\Theta$-category, $U: \DAlg(\CC) \to \CAlg(\CC)$ commutes with arbitrary colimits, hence filtered ones.
    Thus $LU(A)$ is compact in $\DAlg(\CC)$.
    Since $A$ is a localization, Proposition \ref{prop: preservation of epimorphisms} shows $LU(A) \to A$ is an equivalence. 
\end{proof}

Note that compact objects are different in $\DAlg$ and $\CAlg$ in general: 

\begin{example}
    The algebra $\LSym_{\Fp}(\Fp)=\Fp[x]$ is compact in $\DAlg_{\Fp}$ but not in $\CAlg_{\Fp}$. By \cite[Theorem 7.4.3.18]{HA}, it suffices to show that the $\Einfty$-cotangent complex $L^{\Einfty}_{\Fp[x]/\Fp}$ is not a perfect $\Fp[x]$-module.
    By \cite[Proposition 3.2]{RR04}, 
    \[L^{\Einfty}_{\Fp[x]/\Fp} \tensor{\Fp[x]} \Fp
    \simeq \ZZ\tensor{\mathbb{S}} \Fp \] where $\mathbb{S}$ is the sphere spectrum. It is well known that the right-hand side has infinitely generated homotopy groups over $\Fp$\footnote{This phenomenon witnesses the non-smoothness of the `flat' affine line over $\Fp$ in spectral algebraic geometry.}.
\end{example}

\subsection{The $\Theta$-category of quasi-coherent sheaves on a prestack}

The usual symmetric monoidal category $\QCoh(X)$ of quasi-coherent sheaves on a (derived) prestack $X$ has a natural refinement to a $\Theta$-category, constructed in \cite[Remark 2.3]{NT25}. We recall its construction here.

Recall from Example \ref{ex: theta-category of Z-modules} that $(\Mod_\ZZ, \LSym_{\ZZ})$ is a $\Theta$-category.
By Proposition \ref{prop: module cat adjunction},
the functor
\[ \DAlg(\ZZ) \to \thetaCat_{\Mod_\ZZ /}\]
assigning $A$ to the module $\Theta$-category $\ModL_A := \Mod_A(\Mod_{\ZZ}^{\LSym_{\ZZ}})$ is fully faithful.
Restricting to animated rings, we obtain a functor
\[ \QCoh^{\LSym}: \Aff^{op} \to \thetaCat_{\Mod_{\ZZ}/}\]
satisfying 
\begin{enumerate}
  \item the underlying symmetric monoidal category of $\QCoh^{\LSym}(\Spec A)$ is $\Mod_A$;
  \item the monad $M$ on $\QCoh^{\LSym}(\Spec A)$ is $\LSym_A$, the right-left extension of $\LSym$ from $\Mod_A^{\leq 0}$;
  \item the functor $\QCoh^{\LSym}: \Aff^{op} \to \thetaCat_{\Mod_{\ZZ}/}$ is a fully faithful left adjoint.
\end{enumerate}

Since $\thetaCat$ has small limits, the right Kan extension from affine schemes to (accessible) prestacks
\[ \QCoh^{\LSym}: \PreStk \to \thetaCat_{\Mod_\ZZ/}\]
exists and sends colimits to limits.
By Lemma \ref{lemma: theta-categories have small limits}, $\thetaCat \to \CAlg(\PrL)$ preserves limits, so the symmetric monoidal category underlying $\QCoh^{\LSym}(X)$ is naturally identified with $\QCoh(X)$ as defined in \cite[Chapter 1, §1.1]{GR17I}.
Write $\DAlg(X) := \DAlg(\QCohL(X))$.

As in \cite[Chapter 1, 1.5.1]{GR17I}, the underlying category $\QCoh$ carries a $t$-structure defined by declaring that $F \in \QCoh(X)^{\leq 0}$ if and only if whenever $x: S \to X$ for $S \in \Aff$, we have $x^*F \in \QCoh(S)^{\leq 0}$. By construction, the monad $\LSym$ takes $\QCoh(X)^{\leq 0}$ into $\QCoh(X)^{\leq 0}$.

\begin{lemma}\label{lem: flat descent for dalg}
    The $\mathrm{Cat}^\Theta$-valued functor on animated rings $A\mapsto \ModL_A$ satisfies flat descent.
\end{lemma}
\begin{proof}
    By Lemma \ref{lemma: theta-categories have small limits}, the functor $\Cat^\Theta \to \PrL$ sending a $\Theta$-category to its underlying category is conservative and preserves limits. Thus, flat descent follows from flat descent for $R \mapsto \Mod_R$ \cite[Corollary D.6.3.3]{Lurie-SAG}.
\end{proof}

\begin{corollary}\label{cor: flat descent for QCohL}
    Let $f: Y\to X$ be an epimorphism of stacks in the fpqc topology.
    Then $f^*$ induces an equivalence of $\Theta$-categories 
    \[\QCohL(X)\overset{\sim}{\to}\mathrm{Tot}(\QCohL(Y^\bullet))\] where $Y^\bullet$ is the \v{C}ech nerve of $Y\to X$.
\end{corollary}
\begin{proof}
    This is a consequence of Lemma \ref{lem: flat descent for dalg} and \cite[Proposition 3.48]{Yaylali2022NotesDAG}.
\end{proof}

\section{Quasi-affine morphisms}\label{section: quasi-affine}

\subsection{Properties of quasi-affine morphisms}

\begin{definition}
  A derived scheme $X$ is \emph{quasi-affine} if it is quasi-compact and there exists an open immersion $j: X \to \Spec A$ into a derived affine scheme.

  A morphism of prestacks $X \to Y$ is \emph{quasi-affine} if for every affine scheme $S$ mapping to $Y$, the base change $X \times_Y S$ is a quasi-affine scheme.
\end{definition}

\begin{lemma}\label{lemma: composition of quasi-affine morphisms}
    Let $Z$ be a prestack. If $Y \to Z$ and $X \to Y$ are quasi-affine morphisms, then so is $X \to Z$.
\end{lemma}
\begin{proof}
    The proof of \cite[Lemma 2.5.7.2]{Lurie-SAG} works.
\end{proof}

When $f: X \to Y$ is quasi-affine, the category of quasi-coherent sheaves on $X$ is equivalent to modules over $f_*\OO_X$:
\begin{proposition}
\label{prop: QCohL on quasi-affine morphism}
    Let $f: X \to Y$ be a quasi-affine morphism of prestacks.
    \begin{enumerate}
        \item $f_*: \QCoh(X) \to \QCoh(Y)$ is monadic and induces an equivalence of categories
        \[ \QCoh(X) \overset{\sim}{\to} \Mod_{f_*\OO_X}(\QCoh(Y)).\]
        \item $f^*$ induces an equivalence of $\Theta$-categories 
        \[ \Mod_{f_*\OO_X}(\QCohL(Y)) \overset{\sim}{\to} \QCohL(X)\]
        under $\QCohL(Y)$.
    \end{enumerate}
\end{proposition}
\begin{proof}
    Monadicity of $f_*$ follows from \cite[Chapter 3, Proposition 2.2.2]{GR17I}.
    The projection formula is satisfied since $f$ is quasi-compact schematic \cite[Chapter 3, 3.2.4]{GR17I}.
    The monadic equivalence $\QCoh(X) \to \Mod_{f_*\OO_X}(\QCoh(Y))$ is \cite[Chapter 3, 3.3.3]{GR17I}.
    Then Proposition \ref{prop: recognition of module category} shows the equivalence of $\Theta$-categories.
\end{proof}

If $f:X\to Z$ is quasi-affine, the derived ring $f_*\OO_X \in \DAlg(Z)$ classifies maps into $X$ over $Z$:

\begin{lemma}[\cite{luriedagviii}, Lemma 3.2.8]
\label{lemma: maps into quasi-affine morphism}
  Let $f:X \to Z$ be a quasi-affine morphism of prestacks and $g: Y\to Z$ a morphism of prestacks.
  Then the map 
  \[ \Maps_{Z}(Y,X) \to \Maps_{\DAlg(Y)}(g^*f_*\OO_X, \OO_Y)\]
  is an equivalence.
\end{lemma}
\begin{proof}
    Both sides take colimits in $Y$ to limits. So we may assume that $Y$ is an affine scheme. 
    Replacing $Z$ by $Y$ and $X$ by $X \times_Z Y$, we may assume that $X$ is a quasi-affine scheme. Then $X = \widetilde{\Spec}(\Gamma(X,\OO_X))$ \cite[Proposition 4.7]{MM25}, and the statement follows from the adjunction of $\widetilde{\Spec}$ with $\Gamma(-,\OO)$ \cite[Construction 1.2]{MM25}.
\end{proof}

The definition of a quasi-affine morphism is local on the target. However, when the target is a quasi-geometric stack, there is a canonical global choice of quasi-compact open immersion into an affine morphism, arising from the connective cover construction for derived algebras.

\begin{construction}\label{cons: connective cover dalg}
    For a prestack $X$, let $\CAlg(X)^{\leq 0}:= \CAlg(X) \times_{\QCoh(X)} \QCoh(X)^{\leq 0}$ and $\DAlg(X)^{\leq 0} := \DAlg(X) \times_{\QCoh(X)} \QCoh(X)^{\leq 0}$.
    By \cite[Proposition 7.1.3.13]{HA}, colimits of connective $\Einfty$-algebras are connective, so $\CAlg(X)^{\leq 0} \to \CAlg(X)$ preserves colimits. Since $\DAlg(X) \to \CAlg(X)$ is conservative and preserves colimits, $\DAlg(X)^{\leq 0} \to \DAlg(X)$ preserves colimits.
    Thus, these inclusions admit right adjoints $\tau^{\leq 0}: \DAlg(X) \to \DAlg(X)^{\leq 0}$ and $ \tau^{\leq 0}: \CAlg(X) \to \CAlg(X)^{\leq 0}$. If $\AA \in \DAlg(X)$, the counit morphism 
    \[ \tau^{\leq 0} \AA \to \AA\]
    is the \emph{connective cover} of $\AA$.
\end{construction}

The following lemma states that truncation of algebras agrees with truncation on the underlying module:

\begin{lemma}\label{lemma: connective cover and forgetful}
    Let $X$ be a prestack.
    Then the squares 
    \[ 
\begin{tikzcd}[ampersand replacement=\&]
	{\CAlg(X)^{\leq 0}} \& {\CAlg(X)} \\
	{\QCoh(X)^{\leq 0}} \& {\QCoh(X)}
	\arrow[from=1-1, to=1-2]
	\arrow[from=1-1, to=2-1]
	\arrow[from=1-2, to=2-2]
	\arrow[from=2-1, to=2-2]
\end{tikzcd}
\quad \text{ and }\quad 
\begin{tikzcd}[ampersand replacement=\&]
	{\DAlg(X)^{\leq 0}} \& {\DAlg(X)} \\
	{\CAlg(X)^{\leq 0}} \& {\CAlg(X)}
	\arrow[from=1-1, to=1-2]
	\arrow[from=1-1, to=2-1]
	\arrow[from=1-2, to=2-2]
	\arrow[from=2-1, to=2-2]
\end{tikzcd}
    \]
    are right adjointable.
\end{lemma}
\begin{proof}
    The claim means that truncation functors commute with forgetful functors $\DAlg(X) \to \CAlg(X) \to \QCoh(X)$.
    Passing to left adjoints, the claim follows because the left adjoints $\QCoh(X) \to \CAlg(X)$ and $\QCoh(X) \to \DAlg(X)$ send connective objects to connective objects.
\end{proof}

Using connective covers, a quasi-affine morphism over a quasi-geometric stack has a global presentation: 

\begin{lemma}\label{lemma: global quasi-affine into global sections}
    Let $Y$ be a quasi-geometric stack and let $f: X \to Y$ be a prestack over $Y$.
    The following are equivalent:
    \begin{enumerate}
        \item $f: X \to Y$ is quasi-affine;
        \item $X \to \Spec_Y (\tau^{\leq 0} f_*\OO_X)$ is a quasi-compact open immersion;
        \item there exists $A \in \DAlg(Y)^{\leq 0}$ and a factorization $X \to \Spec_Y(A) \to Y$ where $X \to \Spec_Y(A)$ is a quasi-compact open immersion.
    \end{enumerate}
\end{lemma}
\begin{proof}
    Clearly (2) $\implies$ (3) $\implies$ (1).
    Now suppose that $f$ is quasi-affine.
    Let $g: S\to Y$ be a faithfully flat morphism from an affine scheme $S$.
    Since being a quasi-compact open immersion is local in fpqc topology\footnote{For example by the reasons explained in \cite[Proposition 2.8.3.7]{Lurie-SAG} and \cite[Example 6.3.3.6]{Lurie-SAG}, which adapt straightforwardly to the derived setting. 
    }, it suffices to show that $X \times_Y S \to \Spec_S (g^*\tau^{\leq 0} f_*\OO_X)$ is a quasi-compact open immersion.
    Since $g$ is flat, $g^*\tau^{\leq 0} \simeq \tau^{\leq 0}g^*$; since $f$ is schematic quasi-compact, $f_*$ satisfies base change against $g^*$ by \cite[Proposition 2.2.2]{GR17I}. Thus we may assume that $Y = S$ is affine, so that $X$ is a quasi-affine scheme. Then the statement is proved in \cite[Lemma 1.5]{goodmoduli_dag}. 
\end{proof}

\subsection{Classification of quasi-affine maps in derived geometry}\label{ssec: classification of qaff}

As in the work of Bhatt and Halpern-Leistner \cite{BH17}, we would like to characterize the coordinate rings of quasi-affine morphisms purely algebraically.

\begin{definition}
    A prestack $X$ is a \emph{fpqc-algebraic stack} if 
    $X$ is an fpqc sheaf and admits a faithfully flat qcqs surjection from an affine scheme.
\end{definition}

\begin{notation}
\begin{enumerate}
    \item  Let $\QAff(X)$ be the category of quasi-affine morphisms over $X.$ We make this a symmetric monoidal category via the Cartesian monoidal structure.
\item Let $\Opqc(X) \subset \QAff(X)$ be the full subcategory spanned by quasi-compact open immersions into $X$. Note that the Cartesian monoidal structure from $\QAff(X)$ restricts to one on $\Opqc(X)$.
\item If $f: V \to X$ is a quasi-affine morphism, let $\OO_V \in \DAlg(X)$ be the pushforward along $f$ of the structure sheaf of $V$.
\end{enumerate}
\end{notation}

\begin{definition}\label{defn: localization morphism}
  A morphism $A \to B$ in $\DAlg(X)$ is a \emph{localization} if $B$ is a localization in $\DAlg(X)_{A/}$ in the sense of Definition \ref{def: localizations in thetacats}; equivalently, the multiplication map $B \otimes_A B \to B$ is an equivalence.
  A localization $A \to B$ is a \emph{compact-localization} if $B$ is compact as an object of $\DAlg(X)_{A/}$.
\end{definition}

\begin{theorem}\label{theorem: qaff is compact localizations}
  Let $X$ be a fpqc-algebraic stack.
  \begin{enumerate}
    \item $V \mapsto \OO_V$ extends to a fully faithful symmetric monoidal functor $\QAff(X)^{op} \to \DAlg(X)$.
    \item The essential image of $\QAff(X)^{op}\to\DAlg(X)$ consists of bounded-above compact-localizations of connective objects of $\DAlg(X)$.
    \item The essential image of $\Opqc(X)^{op} \to\DAlg(X)$ consists of bounded-above compact-localizations of $\OO_X$.
  \end{enumerate}
  If $\AA\in \DAlg(X)$ is in the essential image of $\QAff(X)^{op}\to \DAlg(X)$ then it is the compact-localization of $\tau^{\leq 0}\AA$ of Construction \ref{cons: connective cover dalg}.
\end{theorem}

Theorem \ref{theorem: qaff is compact localizations} is a version of \cite[Theorem 2.3]{BH17} in derived algebraic geometry instead of spectral algebraic geometry.
The proof of Theorem \ref{theorem: qaff is compact localizations} occupies the rest of this section.
Parts (1), (2), and (3) of Theorem \ref{theorem: qaff is compact localizations} are proved in Lemma \ref{lemma: quasi-affine to coordinate ring is faithful}, Corollary \ref{cor: bounded-above compact localization implies quasi-affine}, and Lemmas \ref{lemma: open substack algebra is compact localization} and \ref{lemma: Bhl full derived theorem}, respectively.
We use the spectral version of Theorem \ref{theorem: qaff is compact localizations} as input to the proof, using our comparison between localizations in $\DAlg$ and $\CAlg$ in §\ref{ssec: localization}.
\begin{lemma}\label{lemma: quasi-affine to coordinate ring is faithful}
  Let $X$ be a prestack. 
  Then $V \mapsto \OO_V$ extends to a fully faithful symmetric monoidal functor $\QAff(X)^{op} \to \DAlg(X)$.
\end{lemma}
The proof follows \cite[Lemma 2.4]{BH17}.
\begin{proof}
  Full faithfulness follows from Lemma \ref{lemma: maps into quasi-affine morphism}.
  It suffices to show that for any two quasi-affine $V_0,V_1\in \QAff(X)^{\mathrm{op}}$ the natural morphism $\OO_{V_0}\otimes \OO_{V_1}\to \OO_{V_0\times_X V_1} $ is an equivalence.
  Pushforward along quasi-affine maps satisfies base change by \cite[Proposition 2.2.2]{GR17I}, so we may assume that $X$ is affine. After this, the proof is the same as that of \cite[Lemma 2.4]{BH17}, noting that limits and colimits of derived algebras remain derived algebras.
\end{proof}


\begin{lemma}\label{lemma: open substack algebra is compact localization}
  Let $X$ be an fpqc-algebraic stack. 
  If $V \in \Opqc(X)$, then $\OO_V$ is an eventually connective compact-localization of $\OO_X$.
\end{lemma}
\begin{proof}
Using the same maneuvers as in the proof of \cite{BH17}, we may reduce to the case when $X$ is an affine scheme. In this case, 
let $uV\to uX$ be the induced morphism of spectral schemes which is also a quasi-compact open immersion.
By Remark \ref{remark: morphisms of theta cats}, $U(\OO_V)=\OO_{uV}$ in the category $\CAlg(X).$ By \cite[Lemma 2.6]{BH17},  $U(\OO_V)$ is an eventually connective compact localization of $\OO_{uX}=U(\OO_X).$ As $U$ preserves the underlying object in $\QCoh(X)$, $\OO_V$ is eventually connective.
 By Proposition \ref{prop: preservation of epimorphisms}, $\OO_V$ is a localization in $\DAlg(X)$; by Lemma \ref{lem: detecting compacts in dalg}, $\OO_V$ is compact in $\DAlg(X)$.
\end{proof}

It now remains to check that any eventually connective compact-localization in $\DAlg(X)$ is of the form $\OO_V$ for an open substack $V \into X$. Following \cite{BH17}, the proof will reduce to the case of a classical Noetherian scheme by spreading out.

\begin{lemma}[Hopkins-Neeman for derived algebras]\label{lem: hopkins-neeman for derived algebras.}
    Let $X$ be a classical Noetherian scheme. Let $\AA\in \DAlg(X)$ be a localization. Then $\AA\simeq \ccolim \OO_{V_i}$ for some pro-object $\{V_i\}$ in $\Opqc(X).$
\end{lemma}
\begin{proof}
    The object $U\AA\in \CAlg(X)$ is a localization, whence by \cite[Lemma 2.9]{BH17} isomorphic to $\ccolim \OO_{V_i}$ for some pro-object $\{V_i\}\in \Opqc(uX).$ 
    Since for any open subscheme $V\in \Opqc(uX)$ there is an open subscheme $V\in \Opqc(X)$, we conclude by Proposition \ref{prop: preservation of epimorphisms}.
\end{proof}

\begin{lemma}\label{lem: affine schemes case of qaff}
    Let $X$ be a classical affine scheme. Let $\AA \in \DAlg(X)$  be a compact-localization of $\OO_X$. Then $\AA\simeq \OO_V$ for some quasi-compact open subscheme $V\subset X$.
\end{lemma}
\begin{proof}
    Let $X=\Spec(S)$ for some discrete ring $S$.  Write $S=\ccolim_i S_i$ as a filtered colimit of finitely generated $\ZZ$-algebras.
    By \cite[Corollary 6.24]{galoisgroup}, if $\Comploc(\DD)$ is the full subcategory of compact localizations in $\DD$,
    \[ \ccolim_{i} \Comploc(\DAlg(S_i)) \overset{\sim}{\to} \Comploc(\DAlg(\ccolim_i S_i)).\]
    Thus we may assume $S$ is a Noetherian ring. By Lemma \ref{lem: hopkins-neeman for derived algebras.}, $\AA\simeq\ccolim_i \OO_{V_i}$ where $\{V_i\}$ is a pro-object in $\Opqc(X)$. 
    Since $\AA$ is compact, the equivalence $\AA \overset{\sim}{\to} \ccolim_i \OO_{V_i}$ factors through some $\AA \to \OO_{V_i}$.
    But by definition of colimit, there is a map $\OO_{V_i}\to A$. Since both algebras are localizations of $S$, these maps are equivalences. 
\end{proof}

\begin{lemma}\label{lemma: Bhl full derived theorem}
    Let $X$ be a fpqc-algebraic stack. Let $\AA\in \DAlg(X)$ be an eventually connective compact-localization of $\OO_X.$ Then $\AA\simeq \OO_V$ for some quasi-compact open substack $V\into X.$
\end{lemma}

\begin{proof}
    If $V \in \Opqc(X)$,
    then $\OO_V$ and $\AA$ are localizations of $\OO_X$, so the space of $\OO_X$-isomorphisms between $\OO_V$ and $\AA$ is empty or contractible. 
    By descent, it suffices to check $\AA \simeq \OO_V$ after passing to an fpqc cover.
    Since $\Opqc(-)$ is also an fpqc sheaf, it suffices to show the existence of such a $V$ on an fpqc cover of $X$. Thus we may assume that $X=\Spec(S)$ is an affine scheme. Write $A=\Gamma(\Spec(S),\AA).$

    Following \cite{BH17}, we use deformation theory. Let $S_n=\tau^{\geq -n}S$ be the $n$-th Postnikov truncation of $S$, so that $S_0=H^0(S)$ is a discrete ring and $\Spec(S_0)$ is a classical affine scheme. We also set $A_n=A\otimes_S S_n$, which are localizations over $S_n$. In fact $A_n \in \DAlg_{S_n}$ is compact for all $n$ since the forgetful functor $\DAlg_{S_n}\to \DAlg_S$ preserves sifted colimits.

    Now we use induction on $n.$  In the case $n=0$, Lemma \ref{lem: affine schemes case of qaff} provides an open subscheme $V_0\subset \Spec(S_0)$ and an isomorphism $A_0\simeq \OO_{V_0}$. 
    As $|\Spec(S)|=|\Spec(S_n)|=|\Spec(S_0)|$, we have unique derived open quasi-compact subschemes $V_n\subset \Spec (S_n)$ and $V\subset \Spec(S)$ providing a lift of $V_0$; indeed we can just restrict the structure sheaf of $\Spec(S_n)$ and $\Spec(S)$ to the open subset $V_0.$

    Thus for each $n\geq 0$ we have compact-localizations $\OO_{V_n}$ and $A_n$ with an isomorphism $A_0\simeq \OO_{V_0}$. We also have compact-localizations $\AA$ and $\OO_V$ in $\DAlg(S).$  Consider the algebras $U(\AA)$ and $U(\OO_V)$ in $\CAlg(S)$ which remain localizations, so by \cite[Lemma 2.5]{BH17}, the $\Einfty$-cotangent complexes $L^{\Einfty}_{U(A)/S}=0$ and $L^{\Einfty}_{U(\OO_V)/S}=0$.
    By \cite[Lemma 2.12]{BH17} (which applies since $\AA$ is eventually connective), for each $n>0$ we get compatible isomorphisms $U(A_n)\simeq U(\OO_{V_n})$ and so by Proposition \ref{prop: preservation of epimorphisms}(3), compatible isomorphisms $A_n\simeq\OO_{V_n}.$

    It remains to check that $A=\lim_{n}A_n $ and $\OO_V=\lim_n \OO_{V_n}$. But this isomorphism can be checked at the level of eventually connective complexes where it is always true.
\end{proof}



We will now explain how to deduce point $(2)$ of Theorem \ref{theorem: qaff is compact localizations} from Lemma \ref{lemma: Bhl full derived theorem}.

\begin{corollary}\label{cor: bounded-above compact localization implies quasi-affine}
    Let $X$ be an fpqc-algebraic stack.
    Then the essential image of $\QAff(X)^{op} \to \DAlg(X)$ consists of bounded-above compact-localizations of connective objects of $\DAlg(X)$.
\end{corollary}

\begin{proof}
First suppose that $\BB$ is a bounded-above compact localization of $\AA \in \DAlg(X)^{\leq 0}$.
Since $V = \Spec_X\AA \to X$ is an affine morphism, $\DAlg(V) \simeq \DAlg(X)_{\AA/}$ by Proposition \ref{prop: recognition of module category}.
Thus $\mathcal{B}$ is a compact localization of $\OO_{V} \in \DAlg(V)$. 
If $S\to X$ is a faithfully flat cover by an affine scheme, then $S \times_X V \to V$ is also a faithfully flat cover by an affine scheme, so $V$ is an fpqc-algebraic stack.
By Lemma \ref{lemma: Bhl full derived theorem}, there is a quasi-compact open immersion $Y\to V$ such that $\OO_Y \simeq \mathcal{B}$ in $\DAlg(V) \simeq \DAlg(X)_{\mathcal{A}/}$. Then the composite $Y \to X$ is quasi-affine.

Conversely, suppose $Y \to X$ is quasi-affine. By Lemma \ref{lemma: global quasi-affine into global sections}, there is an algebra $\AA \in \DAlg(X)^{\leq 0}$ and a factorization $Y \to \Spec_X(\AA) \to X$ where $Y \to \Spec_X(\AA)$ is a quasi-compact open immersion.
By Lemma \ref{lemma: Bhl full derived theorem}, $\OO_Y$ is a bounded-above compact localization of $\AA$.
\end{proof}

\section{Lurie's Tannakian theorems in the derived setting}\label{section: lurie's theorems}

In this section, we prove analogues of Lurie's Tannakian reconstruction results from \cite{luriedagviii}.
The first statement (Theorem \ref{theorem: quasi-affine diagonal and full faithfulness}) is that $f \mapsto f^*$ is fully faithful when the target prestack has quasi-affine diagonal.
When the target is a geometric stack, we identify the essential image in Theorem \ref{theorem: essential image for relative geometric stack} in terms of flat objects.
In both cases, we also prove relative versions over a base prestack $Z$ with quasi-affine diagonal.

\subsection{Full faithfulness}\label{subsection: fullly faithful}

\begin{proposition}\label{prop: reconstruction for quasi-affine morphisms}
  Let $X$ be a prestack.
  Then the functor
  \[ \QAff(X)^{op} \to \thetaCat_{\QCoh^{\LSym}(X)/}\]
  assigning $f: V \to X$ to $\QCoh^{\LSym}(V)$ is fully faithful.
\end{proposition}
\begin{proof}
    By Lemma \ref{lemma: quasi-affine to coordinate ring is faithful}, assigning quasi-affine $V\to X$ to $\OO_V \in \DAlg(X)$ is fully faithful.
    By Proposition \ref{prop: module cat adjunction}, the functor 
    \[ \DAlg(X) \to \thetaCat_{\QCohL(X)/}\]
    sending $\mathcal A$ to $\Mod_{\mathcal A}(\QCohL(X))$ is fully faithful.  
    Finally, the natural morphism $\Mod_{\OO_V}(\QCohL(X)) \to \QCohL(V)$ under $\QCohL(X)$ is an equivalence by Proposition \ref{prop: QCohL on quasi-affine morphism}(2).
\end{proof}

Our first reconstruction result is the derived version of \cite[Proposition 3.3.11]{luriedagviii}.

\begin{theorem}\label{theorem: quasi-affine diagonal and full faithfulness}
    Let $Z$ be a prestack with quasi-affine diagonal.
    Let $X$ be a prestack with quasi-affine diagonal equipped with a morphism $X \to Z$. Let $Y \to Z$ be a morphism of prestacks.
    Then the association $f \mapsto f^*$ gives a fully faithful embedding
    \[ 
        \Maps_Z(Y,X) \to \Fun_\Theta^{\QCohL(Z)}(\QCohL(X),\QCohL(Y)).
    \]
\end{theorem}
\begin{proof}
    Both sides send colimits in $Y$ to limits. Hence we may assume $Y$ is affine,
    so every map $Y \to Z$ and $Y \to X$ is quasi-affine.
    Let $f, f' \in \Maps_Z(Y,X)$.
    Consider the following commutative square:
    \[
\begin{tikzcd}[ampersand replacement=\&,cramped]
	{\Maps_{\PreStk/X}(f,f')} \& {\Fun_\Theta^{\QCohL(X)}(\QCohL(Y),\QCohL(Y))} \\
	{\Maps_{\PreStk/Z}(Y,Y)} \& {\Fun_\Theta^{\QCohL(Z)}(\QCohL(Y),\QCohL(Y))}
	\arrow[from=1-1, to=1-2]
	\arrow[from=1-1, to=2-1]
	\arrow[from=1-2, to=2-2]
	\arrow[from=2-1, to=2-2]
\end{tikzcd}
\]
By Proposition \ref{prop: module cat adjunction}, the mapping categories on the right side of the square are groupoids.
  The horizontal maps are equivalences by Proposition \ref{prop: reconstruction for quasi-affine morphisms}.
  Hence they induce an equivalence between the homotopy fibers of the vertical maps, which proves the claim.
\end{proof}

\begin{remark}\label{rem: full faithfulness and derived affine?}
    The proof of Theorem \ref{theorem: quasi-affine diagonal and full faithfulness} works as long as the diagonal of $Z$ is representable by morphisms $V \to Z$ such that
    \begin{equation*}
        \Maps_Z(V,V) \to \Fun_\Theta^{\QCohL(Z)}(\QCohL(V),\QCohL(V))
    \end{equation*}
    is an equivalence, and similarly for $X$.
    By Proposition \ref{prop: reconstruction for quasi-affine morphisms}, this holds when $V \to Z$ is quasi-affine. In light of \cite[Theorem 3.4]{MM25} and \cite[Theorem 3.16]{MM25}, it is reasonable to expect the same property for most derived affine stacks (in the sense of Mathew--Mondal), although convergence issues make a precise statement slightly delicate.
\end{remark}

Now that Theorem \ref{theorem: quasi-affine diagonal and full faithfulness} is proved, our goal is to calculate the essential image of $f \mapsto f^*$.

\begin{definition}\label{defn: geometric theta-functor}
    Suppose $X \to Z$ and $Y \to Z$ are prestacks.
    Say 
    \[ F \in \Fun_\Theta^{\QCohL(Z)}(\QCohL(X), \QCohL(Y))\]
    is \emph{geometric} if $F$ is equivalent to pullback along some morphism $f \in \Maps_Z(Y,X)$.
\end{definition}

\subsection{Base change and $\Theta$-structure}\label{subsection: base change}

To calculate the essential image in the relative case, we will pass to the case of an affine base scheme using base change for $\Theta$-functors. 
We only work out rudiments of base change along quasi-affine morphisms.

\begin{proposition}\label{prop: pushout of theta cats}
    Let $f:  Z' \to Z$ be a quasi-affine morphism of prestacks and $p: X \to Z$ be any morphism of prestacks.
    Consider the Cartesian square
    \[
\begin{tikzcd}[ampersand replacement=\&,cramped]
	{X'} \& {Z'} \\
	X \& Z
	\arrow["{p'}", from=1-1, to=1-2]
	\arrow["{f_X}"', from=1-1, to=2-1]
	\arrow["\lrcorner"{anchor=center, pos=0.125}, draw=none, from=1-1, to=2-2]
	\arrow["f", from=1-2, to=2-2]
	\arrow["p"', from=2-1, to=2-2]
\end{tikzcd}
    .
    \]
    Then the induced square of pullback functors
    \[
\begin{tikzcd}[ampersand replacement=\&,cramped]
	{\QCohL(Z)} \& {\QCohL(X)} \\
	{\QCohL(Z')} \& {\QCohL(X')}
	\arrow[from=1-1, to=1-2]
	\arrow[from=1-1, to=2-1]
	\arrow[from=1-2, to=2-2]
	\arrow[from=2-1, to=2-2]
\end{tikzcd}
    \]
    is a pushout square in the underlying category of $\thetaCat$. After forgetting to $\CAlg(\PrL)$, the square is the pushout square of \cite[Proposition 3.3.5]{GR17I}.
\end{proposition}
\begin{proof}
Let $\CC \in \thetaCat$.
Consider the natural functor 
\begin{multline}\label{eq: pushout?}
    \Fun_\Theta(\QCohL(X'),\CC) \\
    \to \Fun_\Theta(\QCohL(X),\CC)\underset{\Fun_\Theta(\QCohL(Z),\CC)}{\times} 
    \Fun_\Theta(\QCohL(Z'),\CC).
\end{multline}  
We wish to show that \eqref{eq: pushout?} is an equivalence on underlying groupoids of the mapping categories. 
 By Proposition \ref{prop: QCohL on quasi-affine morphism}, 
\[ \QCohL(X') \simeq \Mod_{(f_X)_*\OO_{X'}}\QCohL(X) \]
and
\[ \QCohL(Z') \simeq \Mod_{f_*\OO_{Z'}}(\QCohL(Z)).\]
Now fix $h \in \Fun_\Theta(\QCohL(X),\CC)$ and take homotopy fibers of the groupoids in \eqref{eq: pushout?} over $h$. 
By Proposition \ref{prop: module cat adjunction}, the homotopy fiber over $h$ of the left-hand side is 
\[ \Hom_{\DAlg(X)}((f_X)_*\OO_{X'}, h_*\one_\CC),\]
while the homotopy fiber of the right-hand side is 
\[ 
    \Hom_{\DAlg(Z)}(f_*\OO_{Z'}, p_*h_*\one_\CC).
\]
The induced map 
\[ \Hom_{\DAlg(X)}((f_X)_*\OO_{X'}, h_*\one_\CC) \to \Hom_{\DAlg(Z)}(f_*\OO_{Z'}, p_*h_*\one_\CC)
\]
is given by applying $p_*$, then pulling back by the morphism
\[f_*\OO_{Z'} \to  f_*p'_*\OO_{X'} \simeq p_*(f_X)_*\OO_{X'}.\]
After passing through the adjunction 
\[ \Hom_{\DAlg(Z)}(f_*\OO_{Z'},p_*h_*\one_\CC) \simeq \Hom_{\DAlg(X)}(p^*f_*\OO_{Z'}, h_*\one_\CC),\]
the induced map is pullback by the base change morphism 
\[ p^*f_*\OO_{Z'} \to (f_X)_*\OO_{X'}\]
of derived algebras.
Since $f$ is quasi-affine, the base change morphism is an equivalence \cite[Chapter II, Proposition 2.2.2]{GR17I}, so we conclude by Proposition \ref{prop: module cat adjunction}.

After forgetting to $\CAlg(\PrL)$, the square is clearly the pushout square of \cite[Proposition 3.3.5]{GR17I}.
\end{proof}

\begin{definition}\label{def: base change of theta}
    Let $f: Z' \to Z$ be a quasi-affine morphism of prestacks.
    Let $p: X \to Z$ and $q: Y\to Z$ be prestacks,
    and let $X' = X \times_Z Z'$, $Y' = Y \times_Z Z'$.
    For  
    \[ 
        F \in \Fun_\Theta^{\QCohL(Z)}(\QCohL(X), \QCohL(Y)),
    \]
    define 
    \[ F \times_{Z} Z' \in 
    \Fun_\Theta^{\QCohL(Z')}(
    \QCohL(X'),\QCohL(Y'))\]
    to be the $\Theta$-functor pushout of $F$ with $\id_{\QCohL(Z')}$.
\end{definition}

\begin{lemma}\label{lemma: check geometric after all base changes}
    Let $Z$ be a prestack with quasi-affine diagonal. Let $X$ be a prestack with quasi-affine diagonal with a map $X \to Z$,
    and let $Y \to Z$ be a prestack.
    For 
    \[F \in \Fun_\Theta^{\QCohL(Z)}(\QCohL(X),\QCohL(Y)),\] 
    the following are equivalent:
    \begin{enumerate}
        \item $F$ is geometric;
        \item for every map from an affine scheme $Z' \to Z$, the base change 
        \[ F \times_{Z} Z' \in \Fun_\Theta^{\QCohL(Z')}(\QCohL(X \times_ZZ'), \QCohL(Y \times_Z Z'))\]
        is geometric.
    \end{enumerate}
\end{lemma}
\begin{proof}
    By definition, if $F \simeq f^*$, then $F \times_Z Z' \simeq (f \times \id_{Z'})^*$.
    Conversely, suppose that $F \times_{Z} Z'$ is geometric whenever $Z' \to Z$ is a map from an affine scheme.
    For each such $Z'$, such a map $f_{Z'}: Y \times_{Z} Z' \to X \times_Z Z'$ is essentially unique by Theorem \ref{theorem: quasi-affine diagonal and full faithfulness}. 
    Since colimits are universal in an $\infty$-topos \cite[Theorem 6.1.0.6]{HTT}\footnote{The key is to write $Y=\ccolim_{Z'\in \Aff, Z'\to Z} Y\times_Z Z'.$},
    \[ \Maps_Z(Y,X) = \lim_{Z'\in \Aff, Z' \to Z} \Maps_{Z'}(Y \times_Z Z', X \times_Z Z').\]
   Therefore there is an essentially unique $f \in \Maps_Z(Y,X)$ such that $(f \times_{Z} Z')^* \simeq F \times_Z Z'$ for all affine schemes $Z'$ mapping to $Z$. It suffices to show that $f^*\simeq F.$
    
    Since $\QCohL$ is right Kan extended from affines, it sends colimits of prestacks to limits, 
    so at the level of underlying groupoids of the categories of $\Theta$-functors,
    \begin{align*}
        \Fun_\Theta^{\QCohL(Z)}&(\QCohL(X), \QCohL(Y))^{\simeq} \\
    &\simeq
    \lim_{\substack{Z' \in \Aff\\Z' \to Z}} \Fun_\Theta^{\QCohL(Z)}(\QCohL(X), \QCohL(Y\times_Z Z'))^{\simeq}
    \\ 
    &\simeq 
    \lim_{\substack{Z' \in \Aff\\Z' \to Z}} \Fun_\Theta^{\QCohL(Z')}(\QCohL(X\times_Z Z'), \QCohL(Y \times_Z Z'))^{\simeq}
    \end{align*}
    where in the last equivalence we have used Proposition \ref{prop: pushout of theta cats}.
    Hence the system of compatible equivalences $(f\times_Z Z')^* \simeq F \times_Z Z'$
    gives an essentially unique equivalence $f^* \simeq F$.
\end{proof}

\subsection{The essential image for geometric stacks}\label{subsection: essential image}

The goal of this section is to calculate the essential image of $f \mapsto f^*$ when the target is a geometric stack. 

The following lemma gives a criterion for a morphism of $\Theta$-categories to be geometric using flat descent. It is a derived version of \cite[Proposition 3.4]{BH17}, which, in turn, follows ideas of Lurie \cite{luriedagviii}. 

\begin{lemma}\label{lemma: bhl recognition geometric}
Let $T$ be an affine scheme.  Let $X$ be a quasi-geometric stack over $T$ and let $S$ be an affine scheme over $T$. For 
    \[F\in \Fun_\Theta^{\QCohL(T)}(\QCohL(X),\QCohL(S)),\] the following are equivalent: 
    \begin{enumerate}
        \item $F$ is geometric;
        \item There exist faithfully flat quasi-affine maps $V\to S$ and $W\to X$ with $W$ affine and a map $\OO_V\to F(\OO_W)$ in $\DAlg(S)$.
    \end{enumerate}
\end{lemma}
\begin{proof}
    This follows mutatis mutandis from \cite[Proposition 3.4]{BH17}, but we indicate the necessary modifications. 
    Clearly (1) implies (2). To prove the converse, assume that $V \to S$ and $W\to X$ are as in (2). 
    Since $W\to X$ is quasi-affine,
    $\QCohL(W) \simeq \Mod_{\OO_W}(\QCohL(X))$.
    Hence $F$ induces a morphism of $\Theta$-categories
    \[ \QCohL(W) \simeq \Mod_{\OO_W}(\QCohL(X)) \to \Mod_{F(\OO_W)}(\QCohL(S)).\]
    On the other hand, the map of derived $\OO_S$-algebras $\OO_V\to F(\OO_W)$ induces a map of $\Theta$-categories 
    \[ 
        \Mod_{F(\OO_W)}(\QCohL(S)) \to \Mod_{\OO_V}(\QCohL(S)) \simeq \QCohL(V).
    \]
    By composition we get a functor $\QCohL(W)\to \QCohL(V)$ which gives us a lift of $F$. Since $W$ is affine, Proposition \ref{prop: module cat adjunction}
    shows that $\QCohL(W) \to \QCohL(V)$ is equivalent to pullback by a morphism $f\colon V\to W$. 
    Arguing similarly for the \v{C}ech nerves $W^\bullet$ of $W\to X$ and $V^\bullet$ of $V\to S$, we get geometric morphisms of cosimplicial objects in $\Theta$-categories $\QCohL(W^\bullet)\to \QCohL(V^\bullet)$.
    By Corollary \ref{cor: flat descent for QCohL},
    the totalizations of $\QCohL(W^\bullet)$ and $\QCohL(V^\bullet)$ are $\QCohL(X)$ and $\QCohL(S)$ respectively. 
    Since $X$ is an fpqc stack, 
    by totalizing, we obtain a geometric morphism 
    $S \to X$ inducing the map of $\Theta$-categories
    \[ 
    F: \QCohL(X)=\mathrm{Tot}(\QCohL(W^\bullet))\to \mathrm{Tot}(\QCohL(V^\bullet)) = \QCohL(S).
    \]
\end{proof}

We also need a relative version of Lemma \ref{lemma: bhl recognition geometric}.

\begin{lemma}\label{lemma: relative recognition geometric}
    Let $Z$ be an fpqc sheaf with quasi-affine diagonal, $X \to Z$ be a relative quasi-geometric stack,
    and $S$ be an affine scheme with a map $S \to Z$.
    For 
    \[F \in \Fun_\Theta^{\QCohL(Z)}(\QCohL(X),\QCohL(S)),\] 
    the following are equivalent:
    \begin{enumerate}
        \item $F$ is geometric;
        \item There exist faithfully flat quasi-affine maps $V \to S$ and $W \to X$ 
        where $W \to Z$ is affine and a map $\OO_V \to F(\OO_W)$ in $\DAlg(S)$.
    \end{enumerate}
\end{lemma}
\begin{proof}
    (1) implies (2). Now suppose that we are given $W \to X, V\to S, \OO_V \to F(\OO_W)$ as in (2). 
    By Lemma \ref{lemma: check geometric after all base changes}, 
    it suffices to show that $F \times_Z Z'$ is geometric for all affine schemes $Z'$ mapping to $Z$. Thus, we may assume $Z$ is an affine scheme.
    Then $X$ is a quasi-geometric stack and $S$ is an affine scheme,
    so by Lemma \ref{lemma: bhl recognition geometric}, there exists $f: S \to X$ such that $F \simeq f^*$.
    As $F$ is compatible with $\QCohL(Z)$, full faithfulness (Theorem \ref{theorem: quasi-affine diagonal and full faithfulness}) shows that $F$ is equivalent to pullback along a map over $Z$.
\end{proof}

We can now calculate the essential image of $f \mapsto f^*$ when the target is a geometric stack: the essential image consists of maps preserving connective objects and flat objects, as in \cite[Theorem 9.3.0.3]{Lurie-SAG} and \cite[Theorem 3.4.2]{luriedagviii}.

\begin{theorem}\label{theorem: essential image for relative geometric stack}
    Let $Z$ be an fpqc sheaf with quasi-affine diagonal,
    $X \to Z$ be a relative geometric stack, and $Y\to Z$ be a morphism of prestacks.
    For 
    \[F \in \Fun_\Theta^{\QCohL(Z)}(\QCohL(X),\QCohL(Y)),\] 
    the following are equivalent:
    \begin{enumerate}
       \item $F$ is geometric;
       \item $F$ preserves connective objects and flat objects.
    \end{enumerate}
\end{theorem}
\begin{proof}
    If $F$ is in the essential image, then $F$ is right t-exact and carries flat objects to flat objects.
    Conversely: by taking colimits in $Y$, it suffices to assume that $Y$ is affine. 
    Since $X \to Z$ is a relative geometric stack, there exists a faithfully flat affine morphism $W \to X$ where $W \to Z$ is affine. 
    Then $\OO_W \in \DAlg(X)$ is connective and faithfully flat,
    so $R' = F(\OO_W) \in \DAlg(\Spec R)$ is connective and faithfully flat.\footnote{Since $F$ preserves cofiber sequences and flat objects, it preserves faithfully flat morphisms by \cite[Lemma D.4.4.3]{Lurie-SAG}.}
    If $V = \Spec R'$, then $V \to Y$ is a faithfully flat affine morphism. Since $\OO_V \simeq F(\OO_W)$ in $\DAlg(Y)$, Lemma \ref{lemma: relative recognition geometric} shows $F$ is geometric.
\end{proof}

When $Z = \Spec(\mathbb{Z})$, we obtain the derived version of \cite[Theorem 3.4.2]{luriedagviii} or \cite[Theorem 9.3.0.3]{Lurie-SAG}.
\begin{corollary}\label{corollary: essential image for geometric stack}
    Let $X$ be a geometric stack and $Y$ a prestack.
    Then for 
    \[ F \in \Fun_\Theta^{\Mod_{\ZZ}^{\LSym}}(\QCohL(X),\QCohL(Y)),\]
    the following are equivalent:
    \begin{enumerate}
        \item $F$ is geometric;
        \item $F$ is right t-exact and carries flat objects to flat objects.
    \end{enumerate}
\end{corollary}

\section{Stacks with compactly generated derived categories}\label{section: compactly generated}

\subsection{Quasi-geometric stacks with compactly generated derived categories}
When $X$ is a quasi-geometric stack such that $\QCoh(X)$ is compactly generated, Corollary \ref{corollary: essential image for geometric stack} can be improved.
In the spectral setting, Bhatt and Halpern-Leistner showed that when $\QCoh(X)$ is compactly generated, the condition in Theorem \ref{theorem: essential image for relative geometric stack} of carrying flat objects to flat objects is automatic \cite[§4]{BH17}. The same is true in the derived setting.


   

\begin{lemma}\label{lemma: compact generation for module cats and dalg}
    Let $X$ be a prestack such that $\QCoh(X)$ is compactly generated. 
    \begin{enumerate}
        \item  For all $\AA\in \DAlg(X)$, the categories $\Mod_{\AA}(\QCoh(X))$ and $\DAlg(X)_{\AA/}$ are compactly generated (in particular $\DAlg(X)$ is compactly generated).

        \item For all morphisms $\AA\to \BB$ in $\DAlg(X)$, the functor $- \otimes_{\AA} \BB: \DAlg(X)_{\AA/}\to \DAlg(X)_{\BB/}$ preserves compact objects.
    \end{enumerate}
   
\end{lemma}
\begin{proof}
 By \cite[Proposition 4.1.10]{Rak20}, if $\mathsf{C}$ is a compactly generated presentable $\infty$-category and $T$ is a sifted-colimit-preserving monad on $\mathsf{C}$, then $\mathsf{C}^T$ is compactly generated. Then (1) follows since $\Mod_\AA(\QCoh(X)) \to \QCoh(X)$ and $\DAlg(X)_{\AA/} \to \Mod_{\AA}(\QCoh(X))$ preserve sifted colimits.

       The right adjoint $\DAlg(X)_{\BB/}\to \DAlg(X)_{\AA/}$ preserves sifted colimits. Indeed sifted colimits in both these categories are computed at the level of modules, and the forgetful functor $\Mod_{\BB}(\QCoh(X))\to \Mod_{\AA}(\QCoh(X))$ preserves sifted colimits. Therefore the left adjoint 
       preserves compact objects, proving (2).
\end{proof}

Following \cite[Lemma 2.3]{BhattAlgebraizationTannaka},
the next lemma shows that if $X$ is a quasi-geometric stack such that $\QCoh(X)$ is compactly generated, $S$ is a quasi-affine scheme, and $F\colon \QCohL(X)\to \QCohL(S)$ is a $\Theta$-functor, the derived geometry of $S$ `localizes' to $\QCohL(X)$ in the sense that $\QCohL(S)$ is a module $\Theta$-category over $\QCohL(X)$.

\begin{lemma}\label{lem: bhatt localize geometry}
Let $X$ be a prestack with quasi-compact schematic diagonal such that $\QCoh(X)$ is compactly generated
and let $S$ be a quasi-affine scheme.
If 
\[F\in \Fun_\Theta(\QCohL(X),\QCohL(S)),\] 
then $F$ admits a colimit-preserving right adjoint $G\colon \QCoh(S)\to \QCoh(X)$ 
so that $F$ induces an equivalence of $\Theta$-categories $$ \Mod_{G(\OO_S)}\QCohL(X)\simeq \QCohL(S)$$ under $\QCohL(X)$. 
Under this equivalence, the functor $F:\QCoh(X) \to \QCoh(S)$ corresponds to base change $M\mapsto M\otimes_{\OO_X} G(\OO_S).$
\end{lemma}
The proof follows \cite[Lemma 2.3]{BhattAlgebraizationTannaka}, but our hypotheses are more general.
\begin{proof}
By Proposition \ref{prop: recognition of module category}, we need to check that the induced right adjoint $G\colon \QCoh(S)\to \QCoh(X)$ satisfies the following properties:
\begin{enumerate}
    \item The functor $G$ is conservative and preserves geometric realizations;
    \item $(F,G)$ satisfies the projection formula, that is, the natural morphism $M\otimes_{\OO_X} G(N)\to G(F(M)\otimes_{\OO_S}N)$ is an equivalence.
\end{enumerate}
Since $F$ is symmetric monoidal and $G$ is right adjoint to $F$,
\begin{align*}
    \Gamma(S,K) = \Hom_{\QCoh(S)}(\OO_S,K)
    = \Hom_{\QCoh(X)}(\OO_X, G(K))
    = \Gamma(X, GK).
\end{align*}
Since $S$ is a quasi-affine scheme, $\Gamma(S,-)$ is conservative (Proposition \ref{prop: QCohL on quasi-affine morphism}). 
Hence $G$ is conservative.

Since $X$ has quasi-compact schematic diagonal, compact objects in $\QCoh(X)$ are dualizable \cite[Chapter 3, Proposition 3.6.8]{GR17I}. The rest of the proof is exactly as in \cite[Lemma 2.3]{BhattAlgebraizationTannaka}.
\end{proof}

We can now prove the derived version of \cite[Theorem 4.1]{BH17}:
\begin{theorem}\label{thm: compactly gen over affine}
Let $T$ be an affine scheme. Let $X$ be a quasi-geometric stack over $T$ such that $\QCoh(X)$ is compactly generated and let $Y$ be a prestack over $T$. Then the functor 
    \[\Maps_T(Y,X)\to \Fun_\Theta^{\QCoh(T)^{\LSym}}(\QCohL(X),\QCohL(Y))\]
    assigning $f$ to $f^*$ is fully faithful with image spanned by $\Theta$-functors preserving connective objects.
\end{theorem}

\begin{proof}
By Theorem \ref{theorem: quasi-affine diagonal and full faithfulness}, 
the functor $f \mapsto f^*$ 
is fully faithful. We must show that a morphism $F \in \Fun_\Theta^{\QCohL(T)}(\QCohL(X),\QCohL(Y))$ preserving connective objects is in the essential image.
    Since both sides send colimits in $Y$ to limits, we may assume $Y=S$ is an affine scheme.
    
    Since $X$ is a quasi-geometric stack, there is a faithfully flat quasi-affine morphism $W \to X$ where $W$ is affine.     
    By Lemma \ref{lemma: bhl recognition geometric}, it suffices to show that $F(\OO_W) \in \DAlg(S)$ is the coordinate ring of a faithfully flat quasi-affine morphism $V \to S$. 
    By Theorem \ref{theorem: qaff is compact localizations}(2), $\OO_W$ is a bounded-above compact localization of some connective $\AA\in \DAlg(X)$. Since $F$ is symmetric monoidal, $F(\OO_W)$ is a localization of $F(\AA)$; since $F$ preserves connectivity, $F(\OO_W)$ is bounded-above and $F(\AA)$ is connective.
    Therefore one needs to check the compactness of $F(\OO_W)$ as an $F(\AA)$-algebra.

    It suffices to show that $F: \DAlg(X)_{\AA/}\to \DAlg(S)_{F(\AA)/}$ preserves compact objects. 
    By Lemma \ref{lem: bhatt localize geometry}, there is an equivalence of $\Theta$-categories
    $$\Phi\colon \QCohL(S)\simeq \Mod_{G(\OO_S)}\QCohL(X)$$ 
    under $\QCohL(X)$,
    where $G: \QCoh(S) \to \QCoh(X)$ is a right adjoint to $F$.
    Under the equivalence $\Phi$, the functor $F$ corresponds to the base change $-\otimes G(\OO_S)$, that is, for $M \in \QCoh(X)$,
    $$\Phi(F(M))\simeq M\otimes G(\OO_S).$$ 
    Further, the category $\DAlg(S)_{F(\AA)/}$ is identified with $\DAlg(X)_{\AA \otimes G(\OO_S)/}$,
    and the functor induced by $F$ is identified with the base change $-\otimes G(\OO_S): \DAlg(X)_{\AA/} \to \DAlg(X)_{\AA \otimes G(\OO_S)/}$. By Lemma \ref{lemma: compact generation for module cats and dalg}$(2)$, this functor preserves compact objects.
    Hence $F(\OO_W)$ is a compact $F(\AA)$-algebra.

    Thus, $F(\OO_W)$ is an eventually connective compact localization of a connective object of $\DAlg(S)$.
    By Theorem \ref{theorem: qaff is compact localizations}(2), $F(\OO_W) \simeq \OO_V$ for $V \in \QAff(S)$.
    It remains to check that $V \to S$ is a faithfully flat morphism. This can be detected on $F(\OO_W)$ and so is a fact about the underlying symmetric monoidal categories, whence the last paragraph of the proof of \cite[Theorem 4.1]{BH17} along with their \cite[Claim 4.2]{BH17} works. 
\end{proof}

There is also a relative version of Theorem \ref{thm: compactly gen over affine} over a base prestack with affine diagonal. 
Our strategy is to reduce to the affine case by the following lemma:

\begin{lemma}\label{lemma: t-structure along affine map}
    Let $Z$ be a prestack with affine diagonal, let $X \to Z$ and $Y \to Z$ be morphisms of prestacks, and let 
    \[ F \in \Fun_\Theta^{\QCohL(Z)}(\QCohL(X),\QCohL(Y))\]
    be a $\Theta$-functor preserving connective objects.
    If $T\to Z$ is a map from an affine scheme,
    then $F \times_Z T$ preserves connective objects.    
\end{lemma}
\begin{proof}
    We may assume $Y$ is an affine scheme.
    Let $q:X_T\to X$ and $p:Y_T\to Y$ be the base changes of $X \to Z$ and $Y \to Z$ along $T\to Z$.
    Let $K \in \QCoh(X_T)^{\leq 0}$.
    Since $q$ is an affine morphism, $q_*$ preserves connective objects, so $q_*K \in \QCoh(X)^{\leq 0}$.
    Since $Z$ has affine diagonal, $Y_T$ is an affine scheme, so $p_*$ reflects connective objects. Then
    \[ 
        p_*(F\times_Z T)(K) \simeq F(q_*K) \in \QCoh(Y)^{\leq 0},
    \]
    so $(F \times_Z T)(K) \in \QCoh(Y_T)^{\leq 0}$, as desired.
\end{proof}

We may now prove the main theorem of this section.

\begin{theorem}\label{theorem: relative compactly generated}
    Let $Z$ be an fpqc sheaf with affine diagonal.
    Let $X \to Z$ be a relative quasi-geometric stack such that $\QCoh(X)$ is compactly generated and let $Y\to Z$ be a morphism of prestacks.
    Then the association $f \mapsto f^*$ gives a fully faithful embedding
    \[ 
        \Maps_Z(Y,X) \to 
        \Fun_\Theta^{\QCohL(Z)}(\QCohL(X),\QCohL(Y))
    \]
    with essential image spanned by $\Theta$-functors preserving connective objects.
\end{theorem}
\begin{proof}
    By Theorem \ref{theorem: quasi-affine diagonal and full faithfulness}, $f \mapsto f^*$ is fully faithful.
    Suppose 
    \[F \in \Fun_\Theta^{\QCohL(Z)}(\QCohL(X),\QCohL(Y))\]
    preserves connective objects.
    By Lemma \ref{lemma: check geometric after all base changes}, it suffices to show that $F \times_Z T$ is geometric whenever $T \to Z$ is a map from an affine scheme.
    Let $X_T$ be the base change of $X$ along $T\to Z$.
    Since $\QCohL(X_T)$ is a module category over $\QCohL(X)$, Lemma \ref{lemma: compact generation for module cats and dalg} shows $\QCoh(X_T)$ is compactly generated.
    By Lemma \ref{lemma: t-structure along affine map}, $F \times_Z T$ preserves connective objects.
    Since $X_T$ is a quasi-geometric stack, Theorem \ref{thm: compactly gen over affine} shows $F\times_Z T$ is geometric.
\end{proof}

\subsection{The case of quasi-compact quasi-separated derived algebraic spaces}

In this section, our goal is to improve Theorem \ref{theorem: relative compactly generated} in the case of quasi-compact quasi-separated algebraic spaces.\footnote{We thank Bhargav Bhatt for asking this question.} 

As in \cite[Theorem 1.5]{BhattAlgebraizationTannaka} and \cite[Theorem 9.6.0.1]{Lurie-SAG} for quasi-compact quasi-separated algebraic spaces, right $t$-exactness is automatic, essentially because the category of quasi-coherent sheaves is generated by a single compact object.
Our proof in the case of derived algebraic spaces reduces to the spectral case by passing to the underlying spectral stack.

\begin{definition}\label{def: derived algebraic spaces}
    A \emph{derived Deligne-Mumford stack} $X$ is an \'etale sheaf with a surjective \'etale morphism $S\to X$ where $S$ is a scheme\footnote{We place no condition on the diagonal. See \cite[Remark 4.5.13]{luriethesis}.}. A derived Deligne-Mumford stack is a \emph{derived algebraic space} if for all discrete rings $R$, the space $X(R)$ is discrete.  
    
    A derived algebraic space $X$ is \emph{quasi-compact} if in the \'etale surjection $S\to X$, one can choose $S$ to be an affine scheme. It is \emph{quasi-separated} if the diagonal $X\to X \times X$ is a quasi-compact morphism.
\end{definition}

\begin{lemma}\label{lemma: preservation of algebraic spaces}
    Let $X$ be a quasi-compact quasi-separated algebraic space. Let $uX$ be the underlying spectral stack obtained by left Kan extension along the functor $\aring\to \CAlg_{\ZZ}$ and \'etale sheafification. Then $uX$ is a quasi-compact quasi-separated spectral algebraic space. 
\end{lemma}
\begin{proof}
Let $\ffX$ be the $\infty$-topos underlying $X$ and let $u\ffX$ be the $\infty$-topos obtained by viewing the structure sheaf as an $\Einfty$-algebra.
By \cite[Lemma 2.19]{ChoughBrauer}, $u\ffX$ represents the functor $uX.$ 
By \cite[Lemma 3.46(2,4,7)]{ChoughFormal}, $u\ffX$ is a quasi-compact quasi-separated spectral algebraic space.
\end{proof}

\begin{corollary}\label{cor: algebraic space compactly generated}
    Let $X$ be a quasi-compact quasi-separated algebraic space. Then $\QCoh(X)$ is compactly generated. 
\end{corollary}
\begin{proof}
    By Lemma \ref{lemma: preservation of algebraic spaces}, $uX$ is a quasi-compact quasi-separated spectral algebraic space. By \cite[Proposition 9.6.1.1]{Lurie-SAG}, we see that $\QCoh(uX) = \QCoh(X)$ is compactly generated.
\end{proof}

\begin{lemma}\label{lemma: qaff derived vs spectral}
    Let $X$ be a derived Deligne-Mumford stack. Then $X$ is a quasi-affine scheme if and only if $uX$ is a quasi-affine spectral scheme.
\end{lemma}
\begin{proof}
    By Lemma \ref{lemma: global quasi-affine into global sections}, $X$ is quasi-affine if and only if $X \to \Spec(\tau^{\leq 0}R\Gamma(X,\OO_X))$ is a quasi-compact open immersion. By Lemma \ref{lemma: connective cover and forgetful}, 
    the underlying $\Einfty$-ring of $\tau^{\leq 0}R\Gamma(X,\OO_X)$ is $\tau^{\leq 0}R\Gamma(uX,\OO_{uX})$.
    By \cite[Lemma 3.46(2)]{ChoughFormal}, a morphism of derived Deligne-Mumford stacks $X \to Y$ is quasi-compact if and only if $uX \to uY$ is quasi-compact; also, $X \to Y$ is an open immersion if and only if $uX \to uY$ is an open immersion. 
    Hence, $X \to \Spec(\tau^{\leq 0}R\Gamma(X,\OO_X))$ is a quasi-compact open immersion if and only if 
    $uX \to \Spec(\tau^{\leq 0}R\Gamma(uX,\OO_{uX}))$ is a (spectral) quasi-compact open immersion. By \cite[Proposition 2.4.1.3]{Lurie-SAG}, this occurs if and only if $uX$ is a quasi-affine spectral scheme.
\end{proof}

\begin{lemma}\label{lem: qcqs algsp quasi-geometric}
    Let $X$ be a quasi-compact quasi-separated algebraic space. Then $X$ is a quasi-geometric 
    stack.
\end{lemma}
\begin{proof}
    We need to show two things: first, that the diagonal $X\to X \times X$ is quasi-affine; second, that $X$ is an fpqc sheaf. 
    
    
    To show that $X$ has quasi-affine diagonal, it suffices to show that if $S \to X$ is an étale surjection from an affine scheme, then $Y= S \times_X S$ is a quasi-affine scheme. By Lemma \ref{lemma: preservation of algebraic spaces}, $uX$ is a qcqs spectral algebraic space. By \cite[Lemma 3.46(1)]{ChoughFormal}, $uS \to uX$ is étale, so by \cite[Proposition 3.4.1.3]{Lurie-SAG}, $uY$ is a spectral quasi-affine scheme. By Lemma \ref{lemma: qaff derived vs spectral}, $Y$ is a quasi-affine scheme.

    To show $X$ is an fpqc sheaf, the proof of \cite[Proposition 9.1.4.3]{Lurie-SAG} works verbatim in the derived setting.
\end{proof}

\begin{theorem}\label{thm: algebraic space essential image}
    Let $X$ be a quasi-compact and quasi-separated algebraic space and let $Y$ be any prestack. Then the association $f\mapsto f^*$ gives an equivalence
    $$\Maps(Y,X)\to  \Fun_{\Theta}^{\ModL_{\ZZ}}(\QCohL(X),\QCohL(Y)).$$
\end{theorem}
\begin{proof}
    By Corollary \ref{cor: algebraic space compactly generated} and Lemma \ref{lem: qcqs algsp quasi-geometric}, $X$ is a quasi-geometric stack and $\QCoh(X)$ is compactly generated. By Theorem \ref{theorem: relative compactly generated}, the association $f\mapsto f^*$ is fully faithful with essential image spanned by $\Theta$-functors which preserve connective objects. To conclude, it suffices to show that a left adjoint symmetric monoidal functor $\QCoh(X) \to \QCoh(Y)$ automatically preserves connective objects. We may assume that $Y$ is an affine scheme. After passing to underlying spectral algebraic spaces, this follows from \cite[Theorem 9.6.0.1]{Lurie-SAG}.
\end{proof}

\begin{corollary}\label{cor: ff embedding algsp}
    Let $\AlgSp_{qcqs}\subset \PreStk$ denote the full subcategory spanned by quasi-compact and quasi-separated algebraic spaces. Then the construction $X\mapsto \QCohL(X)$ determines a fully faithful functor $$\AlgSp_{qcqs}\to \thetaCat_{\ModL_{\ZZ}/}.$$ 
\end{corollary}

\section{Locally Noetherian stacks}\label{section: noetherian}

\subsection{Generalities on locally Noetherian quasi-geometric stacks}

Our goal in this section is to collect some facts and notation about locally Noetherian quasi-geometric stacks from \cite[§9.5.2]{Lurie-SAG}. Since the arguments adapt rather cleanly, we will be brief; in particular we claim no originality in this subsection.

Recall that an animated ring $A$ is \emph{Noetherian} if $H^0(A)$ is Noetherian and each $H^i(A)$ is a finitely generated $H^0(A)$-module.

\begin{definition}\label{def: def of locally Noetherian quasi-geometric stack}
    Say that a quasi-geometric stack $X$ is \emph{locally Noetherian} if there exists a faithfully flat map $f\colon \Spec(A)\to X$ where $A$ is a Noetherian animated ring.
\end{definition}


\begin{notation}
    Let $X$ be a locally Noetherian quasi-geometric stack. Then $\Coh(X)^\heartsuit$ is the full subcategory of $\QCoh(X)^\heartsuit$ spanned by almost-perfect complexes. 
\end{notation}

\begin{proposition}[\cite{Lurie-SAG}, Proposition 9.5.2.3]\label{prop: qcoh noetherian for noetherian stack}
    Let $X$ be a locally Noetherian quasi-geometric stack. Then the category $\QCoh(X)^{\leq 0}$ is locally Noetherian in the sense of \cite[Definition C.6.9.1]{Lurie-SAG}. Moreover, an object $F\in \QCoh(X)^\heartsuit$ is Noetherian if and only if $F \in \Coh(X)^\heartsuit.$
\end{proposition}
\begin{proof}
The proof of \loccit{} adapts straightforwardly.
\end{proof}

\begin{proposition}[\cite{Lurie-SAG}, Proposition 9.5.4.6]\label{prop: lurie tor amplitude lemma}
     Let $X$ be a locally Noetherian quasi-geometric stack and let $Y$ be a prestack. Let $F\colon \QCoh(X)\to \QCoh(Y)$ be a symmetric monoidal functor which preserves small colimits and connective objects. If $K\in \QCoh(X)$ has tor-amplitude $\geq 0$ then $F(K)$ has tor-amplitude $\geq 0$.
\end{proposition}
\begin{proof}
The proof depends on two statements: first, Proposition \ref{prop: qcoh noetherian for noetherian stack}; second, by passing to $0$-truncations of derived stacks \cite[Chapter 1, 1.8.5]{GR17I} via \cite[Lemma 9.5.4.5]{Lurie-SAG} (which in turn is insensitive to any derived structure as it is a fact about classical stacks).
\end{proof}

\subsection{Reconstruction for locally Noetherian (quasi-)geometric stacks}

The goal of this section is to improve Corollary \ref{corollary: essential image for geometric stack} in the case of locally Noetherian geometric and quasi-geometric stacks.  Our results are analogues of those in \cite[§5]{BH17}.

As in the spectral setting \cite[Theorem 9.5.4.1]{Lurie-SAG}, when a geometric stack is locally Noetherian, the condition of a $\Theta$-functor preserving flat objects is automatic from preserving connective objects:

\begin{theorem}\label{theorem: noetherian geometric}
    Let $X$ be a locally Noetherian geometric stack and let $Y$ be any prestack. Then the association $f\mapsto f^*$ gives a fully faithful embedding 
    \[
        \Maps(Y,X)\to \Fun_\Theta^{\ModL_{\ZZ}}(\QCohL(X),\QCohL(Y))
    \]
    with essential image spanned by $\Theta$-functors preserving connective objects.
\end{theorem}

\begin{proof}
An object $K\in \QCoh(X)$ is flat if and only if it is connective and of tor-amplitude $\geq 0$. 
    By Proposition \ref{prop: lurie tor amplitude lemma}, if $F$ is a $\Theta$-functor preserving connective objects, then $F(K)$ is connective and of tor-amplitude $\geq 0$. Now apply Corollary \ref{corollary: essential image for geometric stack}.
\end{proof}

Our next goal is to calculate the essential image of $f \mapsto f^*$ when the target is a locally Noetherian quasi-geometric stack. Here it suffices for a functor to preserve connective objects and almost-perfect objects. Our proof follows \cite[Theorem 9.5.4.2]{Lurie-SAG} very closely, so that our only contribution is to keep track of the derived algebra structures. The ideas involved go back to \cite[§5]{BH17}.

First, we need the following lemma, which describes the coordinate ring of a quasi-affine scheme in terms of (derived) localization:

\begin{lemma}\label{lemma: j-loc derived algebras}
Let $A$ be an animated ring and $J\subset H^0(A)$ a finitely generated ideal. Let $L_J\colon \Mod_A\to \Mod_A^\jloc$ be the $J$-localization functor of \cite[Notation 7.2.4.6]{Lurie-SAG}. Let $j\colon V\into \Spec(A)$ be the quasi-compact open subscheme complementary to the vanishing locus of $J.$ Then:
\begin{enumerate}
    \item  $L_J(A)$ admits a canonical structure of a derived $A$-algebra.

    \item The functor $L_J$ lifts to a left adjoint functor $L^{\DAlg}_J\colon \DAlg_A\to \DAlg_A^\jloc$ where the target is the full subcategory of $\DAlg_A$ spanned by derived $A$-algebras whose underlying modules are $J$-local. The canonical morphism $B\to L_J^{\DAlg}(B)$ exhibits the target as the initial $J$-local derived $A$-algebra under $B.$

    \item The localization of derived algebras $\OO_{\Spec A}\to j_*\OO_V$ identifies with $A\to L_J(A).$
\end{enumerate}
\end{lemma}
\begin{proof}
    For item (1): by \cite[Proposition 7.2.4.9]{Lurie-SAG}, $L_J$, viewed as an endo-functor of $\Mod_A$, preserves all colimits, so $L_J(-) \simeq L_J(A) \otimes_A -$.  
    By \cite[Example 7.2.4.3]{Lurie-SAG},
    the pushforward $j_*\colon \QCoh(V)\to \Mod_A$ induces an equivalence $\QCoh(V) \simeq \Mod_A^{\jloc}$ so that $j_*\OO_V\simeq L_J(A)$.
    We endow the target with the derived algebra structure of the source.
For item $(2)$, we define $L^{\DAlg}_J(B):=L_J(A)\otimes_A B$ where the tensor product is one of derived algebras. It is clear that the derived algebra $L^{\DAlg}_J(B)$ is $J$-local. For the necessary adjoint property, one uses \cite[Proposition 5.2.7.4]{HTT} to check that for any $J$-local derived algebra $B$, the natural map $B\to L^{\DAlg}_J(B)$ is an isomorphism. This is clear because it is an isomorphism at the level of modules.
Item $(3)$ is now evident. 
\end{proof}

\begin{theorem}\label{theorem: noetherian quasi-geometric}
    Let $X$ be a locally Noetherian quasi-geometric stack and let $Y$ be any prestack. Then the association $f\mapsto f^*$ gives a fully faithful embedding 
    \[
        \Maps(Y,X)\to \Fun_\Theta^{\Mod_{\ZZ}^{\LSym}}(\QCohL(X),\QCohL(Y))
    \]
    with essential image spanned by $\Theta$-functors preserving connective objects and almost perfect objects.
\end{theorem}

\begin{proof}
Since both sides take colimits in $Y$ to limits, we may assume $Y$ is an affine scheme.
    Suppose $F\in \Fun_\Theta^{\Mod_{\ZZ}^{\LSym}}(\QCohL(X),\QCohL(Y))$ preserves connective objects and almost perfect objects. Our goal is to show that $F$ is geometric.

    Since $X$ is a quasi-geometric stack, there is a faithfully flat quasi-affine map $V\to X$ where $V$ is affine. Let $\OO_V$ be the pushforward of the structure sheaf of $V$ to $X$. By Lemma \ref{lemma: bhl recognition geometric} it suffices to show that $F(\OO_V)=\OO_{Y'}$ for some faithfully flat quasi-affine map $Y'\to Y$. 
    
    If $\overline{V}=\Spec_X(\tau^{\leq 0}\OO_V)$,
    so that $\OO_{\overline{V}} = \tau^{\leq 0} \OO_V$,
    then by Lemma \ref{lemma: global quasi-affine into global sections}, there is a factorization $V\xrightarrow{j}\overline{V}\to X$ where $j$ is a quasi-compact open immersion and $\overline{V}\to X$ is an affine morphism. 
    Let $J \subseteq H^0(\OO_{\overline{V}})$ be an ideal of definition of the complement of $V \to \overline{V}$.
    Since $X$ is locally Noetherian, Proposition \ref{prop: qcoh noetherian for noetherian stack} applies, so $J$ is a filtered colimit of coherent sheaves $\FF_\alpha\in \mathrm{Coh}(X)^\heartsuit$. 
    Let $J_\alpha \subseteq J$ be the subideal generated by the image of $\FF_\alpha\to J$. The nonvanishing locus of $J_\alpha$ is a subfunctor $V_\alpha$ of $V.$ Since $V$ is quasi-compact and the colimit is filtered, there is an index $\alpha$ so that $V=V_{\alpha}$. 
    Thus, we may replace $J$ by $J_\alpha$ and thus assume that there is $\FF \in \Coh(X)^\heartsuit$ and $i: \FF \to J$ in $\QCoh(X)^\heartsuit$ such that the image of $i$ generates $J$.

Let $R$ be the coordinate ring of $Y$ and let $A$ be the connective derived $R$-algebra $F(\OO_{\overline{V}})$. Set $\overline{Y'}=\Spec(A)$. 
The map $i\colon \FF\to J$ induces a map of $R$-modules 
\begin{equation}\label{eq: noetherian map of generators after F}
    F(\FF)\to F(J)\to F(H^0(\OO_{\overline{V}})).
\end{equation}
Since $F$ is right $t$-exact, $F$ induces an isomorphism 
$H^0(A) = H^0(F(\OO_{\overline{V}})) \overset{\sim}{\to} H^0(F(H^0(\OO_{\overline{V}}))),$
which combined with \eqref{eq: noetherian map of generators after F}
gives a canonical map 
\[ \lambda \colon H^0(F(\FF)) \to H^0(A).\]
The object $\FF\in \Coh(X)^\heartsuit$ is connective and almost perfect, so $F(\FF) \in \Mod_R$ is connective and almost perfect, so $H^0(F(\FF))$ is a finitely presented $H^0(R)$-module. Thus, the image of $\lambda\colon H^0(F(\FF))\to H^0(A) $ generates a finitely presented ideal $K \subseteq H^0(A)$.

Let $Y'\subset \overline{Y'}$ be the quasi-compact open subscheme complementary to the vanishing locus of $K$. By definition, the composite $Y'\to \overline{Y'}\to Y$ is quasi-affine.

It suffices to show the following two facts:
\begin{enumerate}
    \item  $F(\OO_V)\simeq \OO_{Y'}$ in $\DAlg(Y)$, and
    \item $Y'\to Y$ is a faithfully flat morphism.
\end{enumerate}

By Lemma \ref{lemma: j-loc derived algebras}, $\OO_{Y'}$ is the $K$-localization of $A$. By Lemma \ref{lemma: j-loc derived algebras} again, to prove (1), it suffices to show $F(\OO_V)$ is the $K$-localization of $A$. 
There is a fiber sequence $$\tau^{\geq 1}\OO_V[-1]\to \tau^{\leq  0}\OO_V\to \OO_V$$ in $\QCoh(X)$ where the second morphism is a morphism of derived algebras.
Applying $F$, there is a fiber sequence $$F(\tau^{\geq 1}\OO_V[-1])\to A\to F(\OO_V)$$ in $\Mod_R$. 
By the universal property of Lemma \ref{lemma: j-loc derived algebras}, it suffices to show that
\begin{enumerate}[(i)]
    \item The $R$-module $F(\tau^{\geq 1}\OO_V[-1])$ is $K$-nilpotent,
    \item The $R$-module $F(\OO_V)$ is $K$-local.
\end{enumerate}
Since $\OO_V$ is an eventually connective algebra over $\OO_{\overline{V}}$ and so has finitely many cohomology groups in positive degree, the proof of (i) follows exactly as in the proof of \cite[Theorem 9.5.4.2]{Lurie-SAG}. 
For item (ii), again the proof of \loccit{} works.\footnote{On \cite[803]{Lurie-SAG}, there is a broken reference; the content of the reference is \cite[Lemma 5.6]{BH17}.}

After the previous two observations, it suffices to show $Y'\to Y=\Spec(R)$ is a faithfully flat morphism. This again follows as in \loccit{}
\end{proof}

\section{Reconstruction for derived formal prestacks}\label{section: formal}

In this section we explain how to apply Theorem \ref{theorem: essential image for relative geometric stack} to derived formal stacks.
In §\ref{subsec: eqt def formal stacks}, we show that for an animated ring $A$ and a finitely generated ideal $J\subset H^0(A)$, one can define $J$-adic formal $A$-prestacks as prestacks over $\Spf(A)$, recovering the usual definition as presheaves of spaces on $J$-nilpotent $A$-algebras \cite{APCII}. Then we calculate the $\Theta$-structure on $\QCohL(\Spf(A))$ in terms of $J$-complete $A$-modules and derived algebras in §\ref{subsec: theta structure formal modules}. Our main theorem is Theorem \ref{theorem: main thm formal stacks}.

\subsection{Equivalence of definitions of derived formal stacks}\label{subsec: eqt def formal stacks}
\begin{notation}\label{notation: def of formal prestack}
Let $A$ be an animated ring and let $J\subset H^0(A)$ be a finitely generated ideal. 
\end{notation}

\begin{definition}\label{def: definition of nilpotent}
 
An animated $A$-algebra $B$ is called $J$-nilpotent if $H^0(B)$ is annihilated by a power of $J.$
Let $\Nilp_{A,J}\subset \aring_A$ be the full subcategory of $J$-nilpotent $A$-algebras.
\end{definition}

In the literature, for example \cite{APCII}, one finds a definition of $J$-adic formal stacks as presheaves of spaces on $\Nilp_{A,J}$. Our goal is to show that this definition can be efficiently packaged into a relative situation where Theorem \ref{theorem: essential image for relative geometric stack} is applicable.


\begin{definition}[\cite{Lurie-SAG}, §8.1.2]\label{def: spf}
    Let $\Spf(A,J)$ be the subfunctor of $\Spec(A)$ whose value on $B\in \aring$ is the 0-full\footnote{A morphism $X \to Y$ of spaces is 0-full if for all $x \in X$, the morphism $\pi_\ast(X,x) \to \pi_\ast(Y,y)$ is an isomorphism (that is, it is a `union of connected components').}
    subspace of $\Maps_{\aring}(A,B)$ consisting of maps which annihilate some power of $J.$  
    When $J$ is clear from context, we will write $\Spf(A)$ for $\Spf(A,J)$.
\end{definition}
The prestack $\Spf(A)$ can be presented explicitly as an ind-scheme:
\begin{lemma}\label{lemma: spf as ind-scheme}
    Let $A$ be an animated ring and $J \subset H^0(A)$ a finitely generated ideal.
    Then there exists a $\mathbf{Z}_{>0}$-indexed tower
\[
\cdots \longrightarrow A_4 \longrightarrow A_3 \longrightarrow A_2 \longrightarrow A_1
\]
of $J$-nilpotent animated $A$-algebras with the following properties:
\begin{enumerate}
    \item The maps $A_{i+1} \to A_i$ induce a surjection
    \[
    H^0(A_{i+1}) \longrightarrow H^0(A_i)
    \]
    with nilpotent kernel.

    \item The maps $\Spec(A_n) \to \Spf(A)$ induce an equivalence 
    \[ \ccolim_n \Spec(A_n) \overset{\sim}{\to} \Spf(A).\]
    \item Each $A_n$ is perfect as an $A$-module.
\end{enumerate}
\end{lemma}
\begin{proof}
  A proof adapting \cite[Lemma 8.1.2.2]{Lurie-SAG} in the derived setting is given in \cite[Lemma 2.21]{dalg_formal}. See also \cite[Lemma 4.13]{ChoughFormal}.
\end{proof}

\begin{definition}[Adic formal stacks over an animated ring]\label{def: formal stack}
    A prestack $X\colon \aring\to \Spc$ is called a $J$-adic formal $A$-prestack (or a formal $A$-stack if $J$ is clear from context) if we are given a morphism $X \to \Spf(A)$. 
\end{definition}




\begin{lemma}\label{lem: factorization is contractible}
    Let $X$ be a prestack over $\Spec(A)$. Then the space of factorizations
    \[
\begin{tikzcd}[ampersand replacement=\&]
	\& X \\
	{\Spf(A)} \& {\Spec(A)}
	\arrow[dashed, from=1-2, to=2-1]
	\arrow[from=1-2, to=2-2]
	\arrow[from=2-1, to=2-2]
\end{tikzcd}
    \]
    is $(-1)$-truncated, i.e., $\Maps_{\PreStk_A}(X,\Spf(A))$ is either empty or contractible. 
\end{lemma}
\begin{proof}
    Since $\Maps_{\PreStk_A}(X,\Spf(A))$ transforms colimits in $X$ to limits and since a limit of $(-1)$-truncated spaces is $(-1)$-truncated, we are reduced to showing the statement for $X=\Spec(B).$
    The space of such factorizations is the homotopy fiber of $\Spf(A)(B) \to \Spec(A)(B)$, which is $(-1)$-truncated by definition. 
\end{proof}

\begin{proposition}\label{prop: equivalence of def of prestacks}
    The composite of the functor $\PreStk_{/\Spf(A)} \to \PreStk_{/\Spec(A)}$ and the restriction $\PreStk_{/\Spec(A)} \to \PreStk(\Nilp_{A,J})$ is an equivalence of categories
    \[ 
        \Res: \PreStk_{\Spf(A)} \overset{\sim}{\to} \PreStk(\Nilp_{A,J}).
    \]
\end{proposition}

\begin{proof}
We will describe the inverse functor of $\mathrm{Res}$. 
Let $i\colon \Nilp_{A,J}\into \aring_A$ be the inclusion.
Then 
\[i^*\colon \Fun(\aring_A,\Spc)\to \Fun(\Nilp_{A,J}, \Spc)\]
has a habitual left adjoint $i_!\dashv i^*$ given by left Kan extension.
Note that $i^*$ and $i_!$ preserve the full subcategories of accessible functors.
Now if $X \in \PreStk(\Nilp_{A,J})$, then $i_!(X) \to \Spec(A)$ factors through $\Spf(A)$.
For we can write $X$ as a colimit of representable functors, so that
\[ i_!X = \ccolim_{x\in X(B), B \in \Nilp_{A,J}} \Spec B.\]
Since $\Spf(A) \to \Spec(A)$ is $(-1)$-truncated (Lemma \ref{lem: factorization is contractible}), the space of factorizations of $X \to \Spec(A)$ through $\Spf(A)$ is the limit over all points of $X$ of a contractible space, hence exists and is unique up to contractible ambiguity.

We have constructed an adjunction
\[
i_!: \PreStk(\Nilp_{A,J}) \rightleftarrows \PreStk_{/\Spf(A)} : i^*.
\]
Note that the functor $i^*$ also preserves colimits since it also has a right adjoint (right Kan extension). To check that the adjunction is an equivalence, it thus suffices to check for representable functors. For $\Spec(B) \in \PreStk(\Nilp)$ this is trivial; for $\Spec(B) \in \PreStk_{/\Spf(A)}$, $B$ is $J$-nilpotent and hence this is also trivial.
\end{proof}

\subsection{$\Theta$-categories of quasi-coherent sheaves on formal stacks}\label{subsec: theta structure formal modules}

\begin{notation}
   Let $A$ and $J\subset H^0(A)$ be as in Notation \ref{notation: def of formal prestack}. Then $\Mod_A^\jcomp$ is the full subcategory of $\Mod_A$ spanned by $J$-complete $A$-modules \cite[§7.3]{Lurie-SAG}.
   Similarly $\DAlg_A^\jcomp$ is the full subcategory of $\DAlg_A$ consisting of derived algebras whose underlying modules are $J$-complete \cite[§2.1]{dalg_formal}.
\end{notation}

In \cite{dalg_formal}, the following results are proved (the first of which is presumably well known\footnote{For example, see \cite[\href{https://stacks.math.columbia.edu/tag/0H0G}{Tag 0H0G}]{stacks-project} for the statement in a somewhat classical setting, the arguments of which directly inspired the proof in \cite{dalg_formal}.}
):

\begin{theorem}[\cite{dalg_formal}, Theorem 2.26]\label{thm: qcoh of spf}
    There are canonical equivalences of categories
    \begin{enumerate}
        \item $\QCoh(\Spf(A))=\Mod_A^\jcomp$, and
        \item $\DAlg(\Spf(A))=\DAlg_A^\jcomp$.
    \end{enumerate}
\end{theorem}

Our first goal is to package Theorem \ref{thm: qcoh of spf} into a theorem about $\Theta$-categories. 
To do so, we need to put a $\Theta$-structure on $\Mod_A^\jcomp$. This was implicitly done in \cite[§2.1]{dalg_formal}, but we make the $\Theta$-structure explicit here.

\begin{proposition}\label{prop: theta on modjcomp}
    There is a sifted-colimit-preserving monad $(\LSym_A)^\wedge_J$ acting on $\Mod_A^\jcomp$ so that $\DAlg_A^\jcomp$ is precisely the category of algebras over $(\LSym_A)^\wedge_J$, and the forgetful functor $\DAlg_A^\jcomp\to \CAlg_A^\jcomp:=\CAlg(\Mod_A^\jcomp)$ preserves all colimits.
\end{proposition}

\begin{proof}
    The monad $(\LSym_A)^\wedge_J$ is the completion of the monad $\LSym_A$ on $\Mod_A.$ That it is a monad and its algebras are precisely $\DAlg_A^\jcomp$ is \cite[Proposition 2.16]{dalg_formal}. Since $\LSym_A$ preserves sifted colimits and colimits in $\Mod_A^\jcomp$ are computed in $\Mod_A$ and then completed, the monad $(\LSym_A)^\wedge_J$ preserves sifted colimits. 
    By \cite[Proposition 2.13]{dalg_formal}, the functor $\DAlg_A^\jcomp\to \CAlg_A^\jcomp$ preserves all colimits.
\end{proof}

\begin{notation}
    Let $(\Mod_A^{\LSym})^\jcomp$ be the induced $\Theta$-structure on $\Mod_A^\jcomp$.
\end{notation}

\begin{remark}[Compatibility with $\Mod_A^{\LSym}$]\label{remark: compatibility with modl_A}
There is a natural induced morphism of $\Theta$-categories $\ModL_A\to (\ModL_A)^\jcomp$ which is given on the underlying categories in $\PrL$ by $J$-completion $(-)^\wedge_J\colon \Mod_A\to \Mod_A^\jcomp$ and by inclusion on the algebra categories  \cite[Construction 2.12]{dalg_formal}.

\end{remark}

\begin{theorem}\label{theorem: formal theta structures}
    There is a canonical equivalence of $\Theta$-categories
    $$\QCohL(\Spf(A))\simeq (\ModL_A)^\jcomp.$$
\end{theorem}
\begin{proof}
    By Lemma \ref{lemma: spf as ind-scheme}, $\Spf(A) = \ccolim_n \Spec(A_n)$ where $A_n$ are certain $J$-nilpotent animated $A$-algebras.
    Since $\QCohL$ is defined by right Kan extension from affines, 
    \[
        \QCohL(\Spf(A))=\lim_n \QCohL(\Spec(A_n)) = \lim_n \Mod_{A_n}^{\LSym}.
    \]
    Pulling back along $j: \Spf(A) \to \Spec(A)$ gives a $\Theta$-functor 
    \[
        j^*: \Mod_A^{\LSym} \to \lim_n \Mod_{A_n}^{\LSym}.
    \]
    Passing to the right adjoint functor $j_*$, we have left-adjointable squares
    \[
\begin{tikzcd}[ampersand replacement=\&]
	{\lim_n \DAlg(A_n)} \& {\DAlg(A)} \\
	{\lim_n \CAlg(A_n)} \& {\CAlg(A)} \\
	{\lim_n \Mod_{A_n}} \& {\Mod_A}
	\arrow[from=1-1, to=1-2]
	\arrow[from=1-1, to=2-1]
	\arrow[from=1-2, to=2-2]
	\arrow[from=2-1, to=2-2]
	\arrow[from=2-1, to=3-1]
	\arrow[from=2-2, to=3-2]
	\arrow[from=3-1, to=3-2]
\end{tikzcd}
    \]
    The composite $j_*j^*: \Mod_A \to \lim_n \Mod_{A_n} \to \Mod_A$ is the projection onto $\Mod_A^{\jcomp}$ \cite[Theorem 2.26]{dalg_formal}.
    Since the diagram above commutes, the functor 
    $\lim_n\CAlg(A_n) \to \CAlg(A)$ factors uniquely through $\CAlg(A)^{\jcomp}$, and the functor
    $\lim_n \DAlg(A_n)\to \DAlg(A)$ factors uniquely through $\DAlg(A)^{\jcomp}$ over $\CAlg(A)^\jcomp$. 
    This gives a morphism of $\Theta$-categories $\Mod_A^{\jcomp} \to \QCohL(\Spf A)$.
    The underlying functor is an equivalence, whence it is an equivalence.
\end{proof}

\begin{lemma}\label{lemma: spf has affine diagonal}
    Let $A$ be an animated ring and $J \subseteq H^0(A)$ be a finitely generated ideal.
    Then $\Spf(A)$ has affine diagonal.
\end{lemma}
\begin{proof}
    The square 
    \[
\begin{tikzcd}[ampersand replacement=\&,cramped]
	{\Spf(A)} \& {\Spf(A) \times\Spf(A)} \\
	{\Spec(A)} \& {\Spec(A) \times\Spec(A)}
	\arrow[from=1-1, to=1-2]
	\arrow[from=1-1, to=2-1]
	\arrow[from=1-2, to=2-2]
	\arrow[from=2-1, to=2-2]
\end{tikzcd}
\]
is Cartesian.
For an $R$-point of the pullback is the full subfunctor of $\Spec(A)$ on points $A \to R$ such that $J \otimes H^0(A) + H^0(A) \otimes J$ has nilpotent image under $A \otimes A \to A \to R$. But the image of $J \otimes H^0(A) + H^0(A) \otimes J$ under $A \otimes A \to A$ is exactly $J$, so a point of the pullback is exactly a point of $\Spec(A)$ such that the image of $J$ in $H^0(R)$ is nilpotent, that is, a point of $\Spf(A)$.

The claim follows since $\Spec(A) \to \Spec(A) \times\Spec(A)$ is affine.
\end{proof}

As a consequence, we have the following theorem for formal stacks, which applies to all quotients of affine formal schemes by affine formal group schemes.

\begin{theorem}\label{theorem: main thm formal stacks}
    Let $X\to \Spf(A)$ be a relative geometric stack and let $Y\to \Spf(A)$ be any morphism of prestacks. Then $$F\in \Fun_\Theta^{(\Mod_A^{\LSym})^\jcomp}(\QCohL(X),\QCohL(Y))$$ is geometric if and only if it preserves flat objects and connective objects.
\end{theorem}

\begin{proof}
Since $\Spf(A)$ has an affine diagonal by Lemma \ref{lemma: spf has affine diagonal}, to apply Theorem \ref{theorem: essential image for relative geometric stack}, we only need to check that the structure morphism $X\to \Spf(A)$ induces an action of $(\Mod_A^{\LSym})^{\jcomp}$ on $\QCohL(X).$ This follows immediately from Theorem \ref{theorem: formal theta structures}.
\end{proof}

When $X \to \Spf(A)$ has compactly generated derived category, Theorem \ref{theorem: relative compactly generated} gives an improvement to  Theorem \ref{theorem: main thm formal stacks}.

\begin{theorem}\label{theorem: formal compactly generated}
    Let $X\to \Spf(A)$ be a relative quasi-geometric stack such that $\QCoh(X)$ is compactly generated and let $Y \to \Spf(A)$ be a morphism of prestacks. Then
    $$F\in \Fun_\Theta^{(\Mod_A^{\LSym})^\jcomp}(\QCohL(X),\QCohL(Y))$$ 
    is geometric if and only if $F$ preserves connective objects.
\end{theorem}

\begin{example}
Let $X\to \Spf(A)$ be a formal stack. While it is almost never true that $X$ is perfect in the sense of \cite{BFN10}, i.e., $\QCoh(X)$ is almost never $\mathrm{Ind}(\mathrm{Perf}(X))$\footnote{Compact objects on a formal stack are supported on finite schematic thickenings. For example, $\ZZ_p$ is not compact in $\QCoh(\Spf(\ZZ_p)).$}, it is often true that $\QCoh(X)$ is compactly generated for classical stacks such as $X=\mathbb{A}^1/\mathbb{G}_m\times \Spf(A)$. 
\end{example}

\begin{example}
    In a related direction, if $(A,I)$ is a prism in the sense of \cite{BS19}, then the relative prismatization of any $p$-complete animated $A/I$-algebra $R$ is a $(p,I)$-adic formal stack $(R/A)^{\Prism}$ over $\Spf(A)$ endowed with its $(p,I)$-adic topology. When $\Omega^1_{R/(A/I)}/p$ is finitely generated, a result of Bhatt \cite[Corollary 4.21]{MM25} identifies $\QCoh((R/A)^\Prism)$ with $(p,I)$-complete modules over the prismatic cohomology ring $\Prism_{R/A}$. The latter category is compactly generated by
    \cite[Proposition 7.3.1.7]{Lurie-SAG} and \cite[Proposition 7.1.1.12]{Lurie-SAG}. Further, if $\Spf(R)$ admits an \'etale map to $\mathbb{A}^n$\footnote{This always happens Zariski locally when $\Spf(R)$ is smooth over $\Spf(A/I).$ }, 
    then $(R/A)^\Prism$ admits a faithfully flat atlas by an affine formal $A$-scheme \cite[Footnote 36]{Bhatt22}, so $(R/A)^\Prism \to \Spf(A)$ is a relative geometric stack.
\end{example}

\begin{remark}
The symmetric monoidal category $\Mod_{A}^\jcomp$ is an idempotent algebra over $\Mod_A$ in $\CAlg(\PrL)$\footnote{This can be proved using the techniques of \cite[Remark A.8]{AKN23}.}. In particular, a $\Mod_A$-algebra $\CC$ is linear over $\Mod_A^{\jcomp}$ if and only if $\Mod_A^{\jloc} \otimes \CC = 0$.
Similarly, one can show that a $\Theta$-functor $F: \Mod_A^{\LSym} \to \CC$ factors through $(\Mod_A^{\LSym})^{\jcomp}$ if and only if $F(\Mod_A^{\jloc}) = 0$, and such factorizations are unique up to contractible ambiguity.
It follows that the induced functor 
\[
\thetaCat_{\Mod_A^\jcomp/}\to \thetaCat_{\Mod_A/}
\]
is fully faithful.
\end{remark}

\printbibliography

\end{document}